\documentclass[11pt]{amsart}
\usepackage{amsmath}
\usepackage{amscd,amsthm,amssymb,amsfonts}
\usepackage{mathrsfs}
\usepackage{dsfont}
\usepackage{stmaryrd}
\usepackage{euscript}
\usepackage{expdlist}
\usepackage{enumerate}

\input xy
\xyoption{all}
\usepackage[OT2,T1]{fontenc}

\theoremstyle{plain}
\newtheorem{thm}{Theorem}[section]

\newtheorem*{thm*}{Theorem}

\newtheorem{lm}[thm]{Lemma}
\newtheorem{cor}[thm]{Corollary}
\newtheorem*{cor*}{Corollary}

\newtheorem{prop}[thm]{Proposition}
\newtheorem*{conj*}{Conjecture}

\theoremstyle{remark}
\newtheorem*{remark}{Remark}
\newtheorem*{thank}{Acknowledgments}

\theoremstyle{definition}
\newtheorem*{defn*}{Definition}

\newtheorem{defn}[thm]{Definition}

\newcommand{\nc}{\newcommand}

\newcommand{\beq}{\begin{equation}}
\newcommand{\eeq}{\end{equation}}
\newcommand{\bpmx}{\begin{pmatrix}}
\newcommand{\epmx}{\end{pmatrix}}
\newcommand{\bbmx}{\begin{bmatrix}}
\newcommand{\ebmx}{\end{bmatrix}}

\newcommand{\beqcd}[1]{\begin{equation*}\label{#1}\tag{#1}}
\newcommand{\eeqcd}{\end{equation*}}

\numberwithin{equation}{section}

\def\makeop#1{\expandafter\def\csname#1\endcsname
  {\mathop{\rm #1}\nolimits}\ignorespaces}
\makeop{Hom}   \makeop{End}   \makeop{Aut}   %\makeop{Isom}
\makeop{Pic} \makeop{Gal}       \makeop{Div} \makeop{Lie}
\makeop{PGL}   \makeop{Corr} \makeop{PSL} \makeop{sgn} \makeop{Spf}
 \makeop{Tr} \makeop{Nr} \makeop{Fr} \makeop{disc}
\makeop{Proj} \makeop{supp} \makeop{ker}   \makeop{Im} \makeop{dom}
\makeop{coker} \makeop{Stab} \makeop{SO} \makeop{SL} \makeop{SL}
\makeop{Cl}    \makeop{cond} \makeop{Br} \makeop{inv} \makeop{rank}
\makeop{id}    \makeop{Fil} \makeop{Frac}  \makeop{GL} \makeop{SU}
\makeop{Trd}   \makeop{Sp} \makeop{Tr}    \makeop{Trd} \makeop{Res}
\makeop{ind} \makeop{depth} \makeop{Tr} \makeop{st} \makeop{Ad}
\makeop{Int} \makeop{tr}    \makeop{Sym} \makeop{can} \makeop{SO}
\makeop{torsion} \makeop{GSp} \makeop{Tor}\makeop{Ker} \makeop{rec}
\makeop{Ind} \makeop{Coker}
 \makeop{vol} \makeop{Ext} \makeop{gr} \makeop{ad}
 \makeop{Gr}\makeop{corank} \makeop{Ann}
\makeop{Hol} %Holomorphic
\makeop{Fitt} \makeop{Mp} \makeop{CAP}
\def\Isom{\ul{\mathrm{Isom}}}

\def\Spec{\mathrm{Spec}\,}

\DeclareMathAlphabet{\mathpzc}{OT1}{pzc}{m}{it}
\DeclareSymbolFont{cyrletters}{OT2}{wncyr}{m}{n}
\DeclareMathSymbol{\SHA}{\mathalpha}{cyrletters}{"58}

\def\makebb#1{\expandafter\def
  \csname bb#1\endcsname{{\mathbb{#1}}}\ignorespaces}
\def\makebf#1{\expandafter\def\csname bf#1\endcsname{{\bf
      #1}}\ignorespaces}
\def\makegr#1{\expandafter\def
  \csname gr#1\endcsname{{\mathfrak{#1}}}\ignorespaces}
\def\makescr#1{\expandafter\def
  \csname scr#1\endcsname{{\EuScript{#1}}}\ignorespaces}
\def\makecal#1{\expandafter\def\csname cal#1\endcsname{{\mathcal
      #1}}\ignorespaces}
\def\doLetters#1{#1A #1B #1C #1D #1E #1F #1G #1H #1I #1J #1K #1L #1M
                 #1N #1O #1P #1Q #1R #1S #1T #1U #1V #1W #1X #1Y #1Z}
\def\doletters#1{#1a #1b #1c #1d #1e #1f #1g #1h #1i #1j #1k #1l #1m
                 #1n #1o #1p #1q #1r #1s #1t #1u #1v #1w #1x #1y #1z}
\doLetters\makebb   \doLetters\makecal  \doLetters\makebf
\doLetters\makescr
\doletters\makebf   \doLetters\makegr   \doletters\makegr

\normalsize

\makeop{Ram} \makeop{Rep} \makeop{mass}

\makeop{Bl}

\def\Zbarp{\ol{\Z}_p}

\def\Zp{\Z_p}

\def\Qbar{\ol{\Q}}

\def\cA{{\mathcal A}}  %automorphic forms

\def\cD{\mathcal D}

\def\cF{{\mathcal F}}  %Hida family

\def\cK{{\mathcal K}}  %imaginary quadratic field

\def\cO{\mathcal O}

\def\cf{{\mathcal f}}
\def\cW{{\mathcal W}}

\def\cX{\mathcal X}

\def\cC{\mathcal C}

\def\cT{\mathcal T}
\def\cY{\mathcal Y}

\def\cU{\mathcal U}

\newcommand{\Z}{\mathbf Z}
\newcommand{\Q}{\mathbf Q}
\newcommand{\R}{\mathbf R}
\newcommand{\C}{\mathbf C}
\newcommand{\A}{\mathbf A}    % for adele
\newcommand{\F}{\mathbb F}

\def\ol{\overline}  \nc{\opp}{\mathrm{opp}} \nc{\ul}{\underline}

\newcommand{\MX}[4]{\begin{bmatrix}
{#1}& {#2}\\
{#3}&{#4}\end{bmatrix} }

\def\cf{\mbox{{\it cf.} }}

\def\Sg{{\varSigma}}  %%CM type

\def\ndivide{\nmid}
\def\x{{\times}}

\def\iso{\simeq}

\def\bksl{\backslash}

\def\lam{\lambda}
\def\sg{\sigma}

\newcommand{\powerseries}[1]{\llbracket{#1}\rrbracket}

\usepackage[utf8x]{inputenc}
\usepackage{amsmath}
\usepackage{amssymb}
\usepackage{amsthm}
\usepackage{amsmath,amscd}
\usepackage[small,nohug,heads=vee]{diagrams}
\diagramstyle[labelstyle=\scriptstyle]
\title[On the non-triviality of the $p$-adic Abel-Jacobi image of generalised Heegner cycles]{On the non-triviality of the $p$-adic Abel-Jacobi image of generalised Heegner cycles modulo $p$,\\
 II: Shimura curves}
\author{Ashay A. Burungale}
\date{\today}
\address{ Department of Mathematics~\\UCLA ~ \\
Los Angeles, CA 90095-1555, USA
}
\email{ashayburungale@gmail.com}
\def\padic{\text{$p$-adic }}

\begin{document}
\maketitle

\begin{abstract}
%Let $p$ be an odd prime split in an imaginary quadratic extension $K/\Q$. 
Generalised Heegner cycles are associated to a pair of an elliptic newform and a 
Hecke character over an imaginary quadratic extension $K/\Q$. 
The cycles live in a middle dimensional Chow group of 
a Kuga-Sato variety arising from an indefinite Shimura curve over the rationals 
and a self product of a CM abelian surface. 
Let $p$ be an odd prime split in $K/\Q$.
We prove the non-triviality 
of the $p$-adic Abel-Jacobi image of generalised Heegner cycles 
modulo $p$ over the $\Z_p$-anticylotomic extension of $K$.
%We prove the non-triviality 
%of the $p$-adic Abel-Jacobi image of generalised Heegner cycles in a Chow group of a fiber product of a Kuga-Sato variety arising from an indefinite Shimura curve
%and a self product of a CM abelian surface modulo $p$ 
%associated to the Rankin-Selberg convolution of an elliptic Hecke eigenform and a theta series over $K$, 
%along the $\Z_p$-anticylotomic extension of $K$. 
The result implies the non-triviality of the generalised Heegner cycles in the top graded piece of the coniveau filtration on the Chow group 
and proves a higher weight analogue of Mazur's conjecture. 
%This gives new instances of the conjecture and 
In the case of two, the result provides a refinement of the results 
of Cornut-Vatsal and Aflalo-Nekov\'{a}\v{r} on the non-triviality of Heegner points over the $\Z_p$-anticylotomic extension of $K$.
%When $f$ is of weight two, the result implies that  
%the normalised formal group logarithm of Heegner points associated to the convolution are $p$-adic units over the $\Z_l$-anticylotomic extension of $K$.\\
\end{abstract}
\tableofcontents
\section{Introduction} 
\noindent 
When a pure motive over a number field is self-dual with root number $-1$, the Bloch-Beilinson conjecture implies the existence of 
a non-torsion homologically trivial cycle in the Chow realisation.  
For a prime $p$,  the Bloch-Kato conjecture implies the non-triviality of the $p$-adic étale Abel-Jacobi image of the cycle. 
A natural question is to further investigate the non-triviality of the $p$-adic Abel-Jacobi image of the cycle.\\
\\
An instructive set up arises from a self-dual Rankin-Selberg convolution 
of an elliptic Hecke eigenform and a theta series over an imaginary 
quadratic extension $K$ with root number $-1$. 
In this situation, a natural candidate for a non-torsion homologically trivial cycle is the generalised Heegner cycle. 
It lives in a middle dimensional Chow group of a fiber product of a Kuga-Sato variety arising from an indefinite Shimura curve
and a self product of a CM abelian surface. 
In the case of weight two, the cycles coincide with the Heegner points. 
Twists of the theta series by $p$-power order anticyclotomic characters of $K$ give rise to an Iwasawa theoretic family of generalised Heegner cycles. 
Under mild hypotheses, we prove the generic non-triviality of the $p$-adic Abel-Jacobi image of these cycles modulo $p$. 
In particular, this implies the generic non-triviality of the cycles in the top graded piece of the coniveau filtration
along the $\Z_{p}$-anticyclotomic extension of $K$.\\
\\
In the introduction, for simplicity we mostly restrict to the case of Heegner points.\\
\\
Let $p$ be an odd prime. We fix two embeddings $\iota_{\infty}\colon \Qbar \to \C$ and $\iota_{p}\colon \Qbar \to \C_p$. 
Let $v_p$ be the $p$-adic valuation induced by $\iota_p$ so that $v_p(p)=1$.\\
%Let $\F$ be an algebraic closure of $\F_p$. Let $W(\F)$ be the Witt ring. 
%and $\cW=\iota_p^{-1}(W(\F))$. 
%Let $\mathfrak{m}_p$ be the maximal ideal of $\overline{\Z}_p$.\\
\\
Let $K/\Q$ be an imaginary quadratic extension as above and $\cO$ the ring of integers. 
As $K$ is a subfield of the complex numbers, we regard it as a subfield of the algebraic closure $\Qbar$ via the embedding $\iota_{\infty}$.
%Let $\tau_{K/\Q}$ be the quadratic character associated to the extension $K/\Q$. 
%and $\cD_{K/\Q}$ the relative different. 
Let $c$ be the complex conjugation on $\C$ which induces the unique non-trivial element of $\Gal(K/\Q)$ via $\iota_{\infty}$. 
We assume the following:\\
\\
{(ord)}  \text{$p$ splits in $K$.}\\
\\
Let $\mathfrak{p}$ be a prime above $p$ in $K$ induced by the $p$-adic embedding $\iota_p$. 
For a positive integer $m$, let $H_{m}$ be the ring class field of $K$ with conductor $m$ and 
$\cO_{m} = \Z+m\cO$ the corresponding order. 
Let $H$ be the Hilbert class field.\\
%For an integral ideal $\mathfrak{n}$ of $K$, 
%we fix a decomposition $\mathfrak{n=n^{+}n^{-}}$ where $\mathfrak{n}^{+}$ (resp. $\mathfrak{n}^{-}$) is only divisible by split (resp. ramified or inert) primes in $K/\Q$. 
%let $H_{\mathfrak{n}}$ be the ring class field of $K$ of conductor $\mathfrak{n}$.\\
\\
Let $N$ be a positive integer such that $p\ndivide N$. 
Let $f$ be an elliptic newform of weight $2$, level $\Gamma_{0}(N)$ and neben-character $\epsilon$. 
Let $N_{\epsilon}|N$ be the conductor of $\epsilon$. 
Let $E_f$ be the Hecke field of $f$ and $\cO_{E_f}$ the ring of integers. Let $\mathfrak{P}$ be a prime above $p$ in $E_f$ induced by the $p$-adic embedding $\iota_p$. 
Let $\rho_{f}:\Gal(\overline{\Q}/\Q)\rightarrow \GL_{2}(\cO_{E_{f,\mathfrak{P}}})$ be the corresponding $p$-adic Galois representation.\\
\\
We fix a decomposition $N=N^{+}N^{-}$ where $N^{+}$ (resp. $N^{-}$) is only divisible by split (resp. ramified or inert) primes in $K/\Q$. 
We assume the following hypotheses:\\
\\
{(H1)} The level $N$ is square-free and prime to the discriminant of $K$.\\
{(H2)} The number of primes dividing $N^{-}$ is positive and even.\\
{(H3)} The conductor $N_{\epsilon}$ divides $N^{+}$.\\
\\
As $K$ satisfies the Heegner hypothesis for $N^{+}$, the integer ring $\cO$ contains a cyclic ideal $\mathfrak{N^{+}}$ of norm $N^{+}$.\\
From now, we fix such an ideal $\mathfrak{N^{+}}$. Let $\mathfrak{N}_{\epsilon}|\mathfrak{N}^{+}$ be the unique ideal of norm $N_{\epsilon}$.\\
\\
%Let ${\bf{N}}$ denote the norm Hecke character of $\Q$ and ${\bf{N_{K}:=N}}\circ N_{\Q}^{K}$ the norm Hecke character of $K$.
%For a Hecke character $\lam$ of $K$, let $\mathfrak{f}_\lam$ (resp. $\epsilon_\lam$) denote 
%its conductor (resp. central character  i.e. $\lam|_{{\bf{A}}_{\Q}^\times}$, where ${\bf{A}}_{\Q}$ denotes the adele ring over $\Q$). 
Let ${\bf{N}}: {\bf{A}}_{\Q}^{\times}/\Q^{\times} \rightarrow \C^\times$ be the norm Hecke character over $\Q$ given by 
$$     
{\bf{N}}(x)=||x||.        
$$
Here $||\cdot||$ denotes the adelic norm. 
Let ${\bf{N}}_{K}:={\bf{N}}\circ N_{\Q}^{K}$ be the norm Hecke character over $K$ for the relative norm $N_{\Q}^{K}$.
For a Hecke character $\lam:{\bf{A}}_{K}^{\times}/K^{\times} \rightarrow \C^\times$ over $K$, let $\mathfrak{f}_\lam$ (resp. $\epsilon_\lam$) denote 
its conductor (resp. the restriction $\lam|_{{\bf{A}}_{\Q}^\times}$). 
We say that $\lam$ is central critical for $f$ if it is of infinity type $(j_{1},j_{2})$ with $j_{1}+j_{2}=2$ and $\epsilon_{\lam}=\epsilon {\bf{N}^{2}}$.\\
%Let $\mathfrak{f}_\chi^{-}|\mathfrak{f}_\chi$ be the ideal only divisible by primes ramified or inert in $K/\Q$. 
%For a place $v|\mathfrak{f}_\chi^{-}$, let $\Delta_{\chi,v}$ be the finite group $\chi(\cO_{v}^{\times})$. 
%Let $|\Delta_{\chi,v}|$ denote its size.\\
\\
Let $b$ be a positive integer prime to $N$. 
Let $\Sg_{cc}(b,\mathfrak{N^{+}},\epsilon)$ be the set of Hecke characters $\lam$ such that:\\
\\
{(C1)} $\lam$ is central critical for $f$,\\
{(C2)} $\mathfrak{f}_{\lam}=b \cdot \mathfrak{N}_\epsilon$.\\
%{(C3)} The local root number $\epsilon_{q}(f,\lam^{-1})=1$, for all finite primes $q$ prime to $N^{-}$ and $\epsilon_{q}(f,\lam^{-1})=-1$, 
%otherwise.\\
\\
Let $\chi$ be a finite order Hecke character such that $\chi{\bf{ N_{K}}} \in \Sg_{cc}(b,\mathfrak{N}^{+},\epsilon)$. 
Let $E_{f,\chi}$ be the finite extension of $E_f$ obtained by adjoining the values of $\chi$.\\
\\
Let $B$ be an indefinite quaternion algebra over $\Q$ of conductor $N^{-}$. 
Let $Sh_{B}$ be the corresponding Shimura curve of level $\Gamma_{1}(N^{+})$ and $J_{B}$ the Jacobian of $Sh_{B}$. 
Let $f_{B}$ be a normalised Jacquet-Langlands transfer of $f$ to $Sh_{B}$ (\cf \S5.1). 
Let $B_{f}$ be the abelian variety associated to $f_{B}$ by the Eichler-Shimura correspondence and $T_{f} \subset E_{f}$ an order 
such that $B_{f}$ has $T_{f}$-endomorphisms. 
Let $\Phi_{f}:J_{B} \rightarrow B_f$ be the associated surjective morphism. 
%Let $f_{\chi}$ be the toric form associated to the pair $(f,\chi)$ as in \S3.1. 
%Here we state that it is of same weight and level as $f$, $p$-integral and generates a subrepresentation of the automorphic representation generated by $f$. 
%Moreover, it is non-zero precisely when the $T_p$-eigenvalue $\bfa_{p}(f)$ is non-zero. 
Let $\omega_{f_{B}}$ be the differential form on $Sh_{B}$ corresponding to 
$f_{B}$. We use the same notation for the corresponding one form on $J_{B}$. 
Let $\omega_{B_{f}} \in \Omega^{1}(B_{f}/E_{f})^{T_f}$ be the unique one form such that $\Phi_{f}^{*}(\omega_{B_{f}})=\omega_{f_{B}}$. 
Here $\Omega^{1}(B_{f}/E_{f})^{T_{f}}$ denotes the subspace of $1$-forms given by 
%stable under the action of the integer ring $\cO_{E_{f}}$.\\
$$
\Omega^{1}(B_{f}/E_{f})^{T_{f}}=\big{\{}\omega \in \Omega^{1}(B_{f}/E_{f})|[\lam]^{*} \omega = \lam \omega, \forall \lam \in T_{f} \big{\}}.
$$
\noindent\\
Let $A_{b}$ be an abelian surface with endomorphisms by $\cO_{b}=\Z+b\cO$, defined over the ring class field $H_b$. 
Let $\iota_{N^{+}}$ denote a level structure of $\Gamma_{1}(N^{+})$-type on $A_{b}[\mathfrak{N^{+}}]$. 
In view of the moduli interpretation of the Shimura curve (\cf \S2.2), 
we thus obtain a point $(A_{b}, \iota_{N^{+}}) \in Sh_{B}(H_{N^{+}b})$. 
Here we suppress additional data needed in the moduli interpretation for simplicity. 
Let $\Delta_{b}\in J_{B}(H_{N^{+}b})$ be the corresponding Heegner point on the modular Jacobian. 
Here we cohomologically trivialise the Heegner point as in Zhang's work, for example \cite{Z}.
We regard $\chi$ as a character $\chi: \Gal(H_{N^{+}b}/K) \rightarrow E_{f,\chi}$. 
Let $G_{b}=\Gal(H_{N^{+}b}/K)$. 
Let $H_{\chi}$ be the abelian extension of $K$ cut out by the character $\chi$.
To the pair $(f,\chi)$, we associate the Heegner point $P_{f}(\chi)$ given by 
\beq
P_{f}(\chi)= \sum_{\sg \in G_{b}} \chi^{-1}(\sg)\Phi_{f}(\Delta_{b}^{\sg}) \space \in B_{f}(H_{\chi}) \otimes_{T_f} E_{f,\chi}.
\eeq 
%An approach to non-triviality of the Heegner points is via the non-triviality of their $p$-adic formal group logarithm. 
To consider the non-triviality of the Heegner points $P_{f}(\chi)$ as $\chi$ varies, we can consider the non-triviality of the 
corresponding $p$-adic formal group logarithm.
The restriction of the $p$-adic formal group logarithm arising from the one form $\omega_{B_{f}}$ gives a homomorphism 
$\log_{\omega_{B_{f}}}:  B_{f}(H_{\chi}) \rightarrow \C_p$. We extend it to $B_{f}(H_{\chi}) \otimes_{T_f} E_{f,\chi}$ by $E_{f,\chi}$-linearity.\\
\\
We now fix a finite order Hecke character $\eta$ such that $\eta {\bf{N_{K}}} \in \Sg_{cc}(1,\mathfrak{N^{+}},\epsilon)$. 
Let $H_{N^{+}p^{\infty}} = \bigcup_{n\geq 0} H_{N^{+}p^n}$ be the ring class field of conductor $N^{+}p^\infty$. 
Let $K_{\infty} \subset H_{N^{+}p^\infty}$ be the anticyclotomic $\Z_p$-extension of $K$.
In the above notation, we have $G_{p^{n}}=\Gal(H_{N^{+}p^n}/K)$.
%and $\Gamma_{\mathfrak{N}^{+}}=\varprojlim \Gamma_{n}$. 
Let $\Gamma$ be the $\Z_p$-quotient of $\varprojlim G_{p^{n}}$. 
Let $\mathfrak{X}_{0}$ be the subgroup of finite order characters of the group $\Gal(K_{\infty}/K) \iso \Z_p$.
Let $\mathfrak{X}$ denote the set of anticyclotomic Hecke characters over $K$ factoring through $\Gamma$. 
%Let $\mathfrak{X}_{0}$ denote the subset of $p$-power order characters of $\Gamma$. 
For $\nu \in \mathfrak{X}_{0}$, let $G(\nu)$ denote the Gauss sum associated to $\nu$ considered as a primitive character. 
We consider the non-triviality of $G(\nu^{-1})\log_{\omega_{B_{f}}}(P_{f}(\eta \nu))$ modulo $p$, as $\nu \in \mathfrak{X}_{0}$ varies. \\
\\
Our result is the following.\\
\\
\\
{{\bf{Theorem A}}}. 
Let the notation be as above. 
Let $f\in S_{2}(\Gamma_{0}(N),\epsilon)$ be an elliptic newform and $\eta$ a finite order unramified Hecke character over $K$ 
such that $\eta {\bf{N_{K}}} \in \Sg_{cc}(1,\mathfrak{N}^{+},\epsilon)$. 
In addition to the hypotheses (ord), (H1), (H2) and (H3), suppose that\\
\\
{(irr)} the residual Galois representation $\rho_{f}$ modulo $p$ is absolutely irreducible.\\
\\
Then, we have\\
$$\liminf_{\nu \in \mathfrak{X}_{0}}v_{p}\bigg{(} G(\nu^{-1})\log_{\omega_{B_{f}}}(P_{f}(\eta \nu))\bigg{)}=0.$$
\\
In particular,  
for $\nu \in \mathfrak{X}_{0}$ with sufficiently large $p$-power order the Heegner points $P_{f}(\eta\nu)$ are non-zero in $B_{f}(H_{\eta\nu}) \otimes_{T_f} E_{f,\eta\nu}$.\\
\\
\\
\\
In fact, we show that the same conclusions hold when $\mathfrak{X}_{0}$ is replaced by any of its infinite subset. 
For analogous non-triviality of the $p$-adic Abel-Jacobi image of generalised Heegner cycles modulo $p$, we refer to \S6.2.\\
\\
The proof of Theorem A is based on the vanishing of the $\mu$-invariant of an anticyclotomic Rankin-Selberg $p$-adic L-function and 
a recent $p$-adic Waldspurger formula due to Brooks.\\
\\
We now describe the result regarding the $\mu$-invariant. 
Associated to the pair $(f,\eta)$, an anticyclotomic Rankin-Selberg $p$-adic L-function $L_{p}(f,\eta) \in \overline{\Z}_{p}\powerseries{\Gamma}$ is constructed in \cite{Br}. 
It is characterised by the interpolation formula
\beq
\widehat{\lam}(L_{p}(f,\eta)) \doteq L(f,\eta\lam{\bf{N_{K}}},0).
\eeq 
Here $\lam$ is an unramified Hecke character over $K$ with infinity type $(m,-m)$ for $m\geq 0$ and $\widehat{\lam}$ its $p$-adic avatar. 
The notation ``$\doteq$" denotes that the equality holds up to well determined periods. 
Here we only mention that the periods crucially depend on the Jacquet-Langlands transfer $f_{B}$ and 
the underlying Shimura curve.\\
\\
Our result regarding the non-triviality of the $p$-adic L-function $L_{p}(f,\eta)$ is the following.\\
\\
\\
{{\bf{Theorem B}}}. 
Let the notation be as above. 
Let $f\in S_{2}(\Gamma_{0}(N),\epsilon)$ be an elliptic newform and $\eta$ a finite order unramified Hecke character over $K$ 
such that $\eta {\bf{N_{K}}} \in \Sg_{cc}(1,\mathfrak{N}^{+},\epsilon)$. 
Suppose that the hypotheses (ord), (H1), (H2), (H3) and (irr) hold. Then, we have
$$ \mu(L_{p}(f,\eta))=0.$$
\\
\\
We now describe the strategy of the proofs. Some of the notation used here is not followed in the rest of the article. 
The characters in $\mathfrak{X}_{0}$ are outside the range of interpolation for the $p$-adic L-function $L_{p}(f,\eta)$ and these values basically equal 
the $p$-adic formal group logarithm of Heegner points. This is based on the $p$-adic Waldspurger formula in \cite{Br}. 
A phenomena of this sort was first found by Rubin in the CM case (\cf \cite{R})  and 
recently by Bertolin-Darmon-Prasanna in the general case (\cf \cite{BDP1}). 
As $\mathfrak{X}_{0}$ is a dense subset of characters of $\Gamma$, 
the $p$-adic Waldspurger formula reduces Theorem A to Theorem B. 
We prove the later based on a strategy of Hida. 
This strategy was introduced in \cite{Hi3}. 
Hida proves the vanishing of the $\mu$-invariant of a class of anticyclotomic Katz $p$-adic L-functions in \cite{Hi3}. 
Let $G_{f,\eta} \in \Zbarp \powerseries{T}$ be the power series expansion of the 
measure $L_{p}(f,\eta)$ regarded as a $p$-adic measure on $\Z_p$ with support in $1+p\Z_p$, i.e. 
\beq 
G_{f,\eta}= \int_{1+p\Z_p} (1+t)^{y} dL_{p}(f,\eta)(y) = \sum_{k \in \Z_{\geq 0}} \Big{(}\int_{1+p\Z_p} \binom {y}{k} dL_{p}(f,\eta)(y) \Big{)} t^{k}. 
\eeq
The starting point is the fact that there are modular forms $(f_{i})_{i=1}^{m}$ on the Shimura curve $Sh_{B}$ such that
\beq 
G_{f,\eta}=\sum_{i=1}^{m} a_i\circ(f_{i}(t)). 
\eeq
Here $f_{i}(t)$ is the
$t$-expansion of $f_{i}$ around a well chosen CM point $x$ corresponding to the trivial ideal class in $\Pic(\cO)$
on the Shimura curve $Sh_{B}$. 
Moreover, $a_i$ is an automorphism of the deformation space of $x$ in $Sh_{B}$ such that 
the $a_{i}$'s are mutually irrational. 
The modular forms $f_{i}$'s are closely related to the Jacquet-Langlands transfer $f_{B}$. 
Based on Chai's study of Hecke-stable subvarieties of a Shimura variety, we prove the linear independence of $(a_i\circ f_{i})_{i=1}^{m}$
modulo $p$. 
The independence is an analogue of Ax-Lindemann-Weierstrass conjecture for the mod $p$ reduction of the Shimura curve $Sh_{B}$. 
The proof relies on Chai-Oort rigidity principle that a Hecke stable subvariety of a mod $p$ Shimura variety is a Shimura subvariety. 
The principle is a mod $p$ analogue of Andre-Oort conjecture for self products of the Shimura curve. 
An analogue of Ax-Lindemann-Weierstrass of this sort was originally found by Hida in the Hilbert modular case (\cf \cite{Hi3}). 
We closely follow Hida's approach. 
The assumption (irr) plays a key role in the independence as it implies the non-constancy of $f_{i}$'s modulo $p$. 
In view of the independence and (1.4), it follows that $\mu(L_{p}(f, \eta))=\min_{i} \mu(f_{i}(t))=\min_{i}\mu(f_{i})$. 
Based on our optimal choice of the Jacquet-Langlands transfer $f_{B}$ and the $p$-integrality criterion in \cite{Pr}, we deduce that $\min_{i}\mu(f_{i})=0$. 
This finishes the proof. 
We would like to emphasize our perception that the independence lies at the heart of the proof of Theorem B.\\
\\
``In particular" part of the Theorem A was conjectured by Mazur in the early 1980's (\cf \cite{M}). 
It was proven by Cornut-Vatsal and Aflalo-Nekov\'{a}\v{r} in the mid and late 2000's, respectively (\cf \cite{C}, \cite{V1}, \cite{CV1}, \cite{CV2} and \cite{AN}). 
We give a new approach and as far as we know the theorem is a first result regarding the non-triviality of the $p$-adic formal group logarithm of the 
Heegner points modulo $p$. 
It seems suggestive to compare our approach with the earlier approach.
%Also, a related $l\neq p$-situation is considered in \cite{V1}. 
%In \cite{V1}, Jochnowitz congruence is a starting point. It reduces the non-triviality of the Heegner points to the non-triviality of the Gross points on 
%a suitable definite Shimura ``variety". This in turn is studied via Ratner's theorem. In this article, the non-triviality is based on the modular curve itself. 
In the earlier approach, Ratner's theorem on ergodicity of torus actions is fundamental. 
As indicated above, our approach fundamentally relies on Chai's theory of Hecke stable subvarieties of a mod $p$ Shimura variety. 
It is rather surprising that we have these quite different approaches for the same characteristic zero non-triviality. 
A speculation along these lines was expressed in \cite{V2}. 
It seems interesting that the ergodic nature of the earlier approach is still present in our approach albeit in a more geometric form. 
For a more detailed comparison, we refer to \cite{Bu4}.
Before the $p$-adic Waldspurger formula, 
the non-triviality and Hida's strategy appeared to be complementary. 
The formula also allows a rather smooth transition to the higher weight case.\\ 
%We give a new proof and as per as we know the theorem is a first result regarding the non-triviality of the $p$-adic formal group logarithm of Heegner points modulo $p$. 
%Also, a related $l\neq p$-situation is considered in \cite{V1}. 
%In our case, Chai's theory of Hecke stable subvarieties of a Shimura variety is fundamental. 
%The results of Cornut-Vatsal and Aflalo-Nekovat thus fit in the framework of Hida. Before the $p$-adic Waldspurger formula, 
%the result and the framework appeared to be complementary. It would be quite interesting to explore whether these quite different approaches are related intimately. 
\\
As far as we know,  
higher weight analogue of Theorem A is a first general result regarding the non-triviality of generalised Heegner cycles. 
In particular, the non-triviality implies that the Griffiths group of the fiber product of a Kuga-Sato variety arising from the Shimura curve
and self product of a CM abelian surface has infinite rank over $\overline{\Q}$ (\cf remark (2) following Theorem 6.2). 
An analogue of such a result for the Griffiths group of the Kuga-Sato variety is due to Besser (\cf \cite{Be}). 
In \cite{Be}, the approach is based on the generic non-triviality of classical Heegner cycles over a class of varying imaginary quadratic extensions. 
Regarding generalised Heegner cycles in the case of modular curves, the only earlier known result seems to be the non-triviality of several examples of such families in \cite{BDP3}. 
Based on the Cornut-Vatsal approach, Howard has proven a related characteristic zero non-triviality of classical Heegner cycles in the case of modular curves in \cite{Ho} and \cite{Ho1}.\\ 
\\
%As per as we know,  
%higher weight analogue of Theorem A is the first result regarding the non-triviality of generalised Heegner cycles.\\ 
%\\
Theorem A and its higher weight analogue has various arithmetic applications. 
The non-triviality of the $p$-adic Abel-Jacobi image along the $\Z_p$-anticyclotomic extension is a typical hypothesis 
while working with an Euler system, for example \cite{N} and \cite{F}. 
The hypothesis typically plays a key role in bounding the size of relevant Selmer groups. 
Cornut-Vatsal's non-triviality results have been applied by Nekov\'{a}\v{r} in his proof of the parity conjecture 
for a class of Selmer groups (\cf \cite{N1}, \cite{N2}, \cite{N3} and 
references therein). 
It seems that the higher weight analogue of Theorem A can have potential applications to the parity conjecture 
for Selmer groups associated to a more general Rankin-Selberg convolution. 
As mentioned before, the non-triviality gives an evidence for the refined Bloch-Beilinson conjecture. 
Some other arithmetic consequences of  the ``In particular" part of Theorem A are well documented in the literature, for example \cite{C}, \cite{V1}, \cite{AN} and \cite{F}.\\
\\
Theorem B is perhaps of independent interest as well. 
It is an input in the ongoing work of Jetchev, Skinner and Wan on the $p$-adic Birch and Swinnerton-Dyer conjecture (\cf \cite{JSW}). 
The result in turn is expected to play a key role in the ongoing work of Bhargava, Skinner and Zhang on the Birch and Swinnerton-Dyer conjecture 
for a large proportion of elliptic curves over the rationals.\\ 
%Theorem B can also be used as an input in the proof of Perrin-Rou's conjecture on Heegner points in \cite{W}, instead of the result on Iwasawa $\mu$-invariants in \cite{Hs3}.\\
\\
Hida's strategy has been influential in the study of non-triviality of various arithmetic invariants modulo $p$. 
For example, the non-triviality results in \cite{Hi3}, \cite{BuHs}, \cite{Hs1}, \cite{Hs3}, \cite{Bu} and \cite{Bu2} are 
based on the Hilbert modular independence in \cite{Hi3}. 
The strategy and especially its geometric aspects have been further explored and refined in \cite{BuHi}.\\
\\
Recently, the $p$-adic Waldspurger formula has been generalised to modular forms on Shimura curves over a totally real field (\cf \cite{LZZ}). 
In \cite{Bu1}, we consider an analogue of the independence for quaternionic Shimura varieties over a totally real field. 
In the near future, we hope to consider an analogous 
non-triviality of generalised 
Heegner cycles over a CM field. 
In a forthcoming article \cite{Bu3}, we consider an $(l,p)$-analogue of the results. 
Such an analogue in the case of modular curves is considered in \cite{Bu2}.\\
\\
The article is perhaps a follow up to \cite{BDP1} and particularly, \cite{Hi3}. We refer to them for a general introduction. 
In the exposition, we often suppose that the reader is familiar with them, particularly \cite{Hi3}.\\
\\
The article is organised as follows. In \S2, we describe generalities regarding the Shimura curve. 
In \S3, we consider the Serre-Tate deformation space of an ordinary closed point on the Shimura curve. 
Some parts of \S3 perhaps logically come before \S2. 
For example, the conclusion of \S3.1 is needed in the beginning of \S2.4. 
We suggest the reader to proceed accordingly. 
For convenience, we maintain the current ordering. 
In \S4, we prove the independence of mod $p$ modular forms. In \S5, we prove the vanishing of the $\mu$-invariant of 
a class of anticyclotomic Rankin-Selberg $p$-adic L-functions. 
In \S6, we prove the non-triviality of the $p$-adic Abel-Jacobi image of generalised Heegner cycles modulo $p$ over the $\Z_p$-anticyclotomic extension.\\
\\
We refer to the survey \cite{Bu4} for an expository account of the article.\\
\begin{thank} 
We are grateful to our advisor Haruzo Hida for continuous guidance and encouragement. 
We are grateful to Christopher Skinner for advice and encouragement. 
The topic was suggested by Haruzo Hida and Christopher Skinner as a follow up to \cite{Bu2}. 
We are grateful to Chandrashekhar Khare for continuous encouragement.
We thank Ernest Hunter Brooks and Francesc Castella for helpful conversations, particularly regarding the $p$-adic Waldspurger formula. 
We thank Dimitar Jetchev and Burt Totaro for detailed comments on the previous versions of the article. 
We thank Henri Darmon, Ben Howard, Ming-Lun Hsieh, Jan Nekov\'{a}\v{r}, Kartik Prasanna, 
Richard Taylor and Xin Wan for interesting conversations about the topic. 
Finally, we are indebted to the referee. 
The current form of the article owes a great deal to the referee's constructive criticism and thorough suggestions.
\end{thank}
\noindent\\
\noindent {\bf{Notation}}
We use the following notation unless otherwise stated.\\
\\
%Let $\cF_+$ denote the totally positive elemnts in $\cF$. We sometime use $O$ to denote the ring of integers $\cO_\cF$. 
%Let $\cD_\cF$ (resp. $D_\cF$) be the different (resp. discriminant) of $\cF/\Q$.  
Let $D_{K}$ be the 
different (resp. discriminant) of $K/\Q$. 
%Let ${\bf{h}}$ (resp. ${\bf{h}}_K$) be the set of finite places of $F$ (resp. $K$).
Let $v$ be a place of $\Q$ and $w$ be a place of $K$ above $v$. Let $\Q_v$ be the completion of $\Q$ at $v$, $\varpi_v$ an uniformiser 
and $K_{v} = \Q_{v} \otimes_{\Q} K$.\\
%For the $p$-ordinary CM type $\Sg$ of $\cK$ as above, let
%$\Sg_p = \{ w \in {\bf{h}}_\cK | w|p$ and $w$ induced by $\iota_p \circ \sigma$, for $\sigma \in \Sg\}$.\\
\\
For a number field $L$, let ${\bf{A}}_{L}$ be the adele ring, ${\bf{A}}_{L,f}$ the finite adeles and ${\bf{A}}_{L,f}^{\Box}$ the finite adeles away from a finite set of places 
$\Box$ of $L$. 
%For $a \in L$, let $\mathfrak{il}_{L}(a)=a(\cO_L \otimes_{\Z} \widehat{\Z}) \cap L$. 
Let $G_L$ be the absolute Galois group of $L$ and $\rec_{L}: {\bf{A}}_{L}^\times \rightarrow G_{L}^{ab}$ the geometrically normalized reciprocity law.\\
%Let $\psi_\Q$ be the standard additive character of ${\bf{A}}_\Q$
%such that $\psi_\Q(x_\infty)=\exp(2\pi i x_{\infty})$, for $x_\infty \in \R$. Let $\psi_L : {\bf{A}}_{L}/ L \rightarrow \C$ be given by 
%$\psi_{L}(y)=\psi_{\Q}\circ (\Tr_{L/\Q}(y))$, for $y \in {\bf{A}}_{L}$. We denote $\psi_\Q$ by $\psi$.\\
\\
\\
\section{Shimura curves}
\noindent In this section, we describe generalities regarding Shimura curves arising from indefinite quaternion algebras over the rationals.\\
\subsection{Setup} In this subsection, we recall a basic setup regarding Shimura curves. We occasionally follow \cite{Br}\\
\\
Let the notation and hypotheses be as in the introduction. 
In particular, $B$ denotes the indefinite quaternion algebra over $\Q$ of conductor $N^{-}$. 
Let $\cO_{B}$ be a maximal order in \cite[\S2.1]{Br}.\\
\\
Let $p_{0}$ be an auxiliary prime such that:\\
\\
(A1) For a prime $l$, the Hilbert symbol $(p_{0},N^{-})_{l}=1$ if and only if $l|N^{-}$.\\
(A2) All primes dividing $pN^{+}$ split in the real quadratic field $M_{0}:=\Q(\sqrt{p_{0}})$.\\
\\
The choice of $p_{0}$ determines a Hashimoto model for $B$ as follows. 
The algebra $B$ has a basis $\{1, t, j, tj\}$ as a vector space with 
$t^{2}=-N^{-}$, $j^{2}=p_{0}$ and $tj= -jt$ (\cf \cite[\S2.1]{Br}). 
Moreover, the $\Z$-span of the basis is contained in a unique maximal order in $\cO_{B}$. 
Let $\dagger$ be an involution on $B$ given by $b \mapsto b^{\dagger}=t^{-1}\overline{b}t$ for $b \in B$. 
Here $b \mapsto \overline{b}$ is the main involution of $B$.\\
\\
Let $G_{/\Q}$ be the algebraic group $B^\times$ and $h_0: Res_{\C/\R} \mathbb{G}_{m} \to G_{/\R}$ be the morphism of real group schemes arising from 
$$ a+bi \mapsto  \MX{a} {-b}{b}{a},$$
where $a+bi \in \C^\times$. Let $X$ be the set of $G(\R)$-conjugacy classes of $h_0$. We have a canonical isomorphism $X \iso \C - \R$. 
The pair $(G,X)$ satisfies Deligne's axioms for a Shimura variety. 
It gives rise to a tower $(Sh_{K}=Sh_{K}(G,X))_K$ of smooth proper curves over $\Q$ indexed by open compact subgroups $K$ of $G(\A_{\Q,f})$. 
%The pro-algebraic variety $Sh_{/\Q}$ is the projective limit of these curves. 
The complex points of these curves are given by\\
\beq 
Sh_{K}(\C) =  G(\Q) \backslash X \times G(\A_{\Q,f}) / K .
%Sh(\C)= G(\Q) \backslash X \times G(\A^f) / \overline{Z(\Q)}.\\
\eeq
\noindent\\
%Here $\overline{Z(\Q)}$ is the closure of the center $Z(\Q)$ in $G(\A^f)$ under the $ad\acute{e}lic$ topology. 
%From (2.1), it follows that $Sh_{/\Q}$ is endowed with an action of $G(\A^{f})$ (\cf \cite[\S4.2]{Hi1}). This gives rise to the Hecke action.\\
%\\
%In what follows, we also use the following notions.\\
In what follows, we consider the case when $K$ arises from the congruence subgroup $\Gamma_{1}(N^{+})$. 
Here $\Gamma_{1}(N^{+})$ denotes the norm one elements in an Eichler order of level $N^{+}$ in $\cO_{B}$.
We use the notation $Sh_{B}$ to denote the corresponding Shimura curve.\\
%defined in the next subsection.\\
\\
\\
\subsection{$p$-integral model} 
In this subsection, we briefly recall $p$-integral smooth models of 
the Shimura curves. We occasionally follow \cite{Br}.\\
%variety $Sh/G(\Z_p)_{/\Q}$.\\
\\
Let the notation and hypotheses be as in \S2.1. 
The Shimura curve $Sh_{B/\Q}$ represents a functor $\cF$ classifying abelian surfaces having multiplication by $\cO_{B}$ 
along with additional structure (\cf \cite[\S2.2]{Br}). A $p$-integral 
interpretation of $\cF$ leads to a $p$-integral smooth model of $Sh_{B/\Q}$. 
A closely related functor $\cF^{(p)}$ gives rise to the relevant Shimura variety of level prime to $N^{-}p$ 
i. e. a tower of Shimura curves of level prime to $N^{-}p$. 
For later purposes, we consider the Shimura variety.\\
\\
The functor $\cF^{(p)}$ is given by\\
$$\cF^{(p)}: SCH_{/\Z_{(p)}} \to SETS$$
\beq S \mapsto \{(A,\iota,\bar{\lam},\eta^{(p)})_{/S}\}/ \sim .\eeq
Here,\\
\\
(PM1) $A$ is an abelian surface over $S$.\\
(PM2) $\iota: \cO_{B} \hookrightarrow \End_{S}A$ is an algebra embedding.\\
(PM3) $\bar{\lam}$ is a $\Z_{(p)}^{+}$-polarisation class of a homogeneous polarisation $\lam$ such that the Rosati involution of $\End_{S}A$ maps $\iota(l)$ to $\iota(l^{\dagger})$
 for $l \in \cO_B$.\\
(PM4) Let $\cT^{(N^{-}p)}(A)$ be the prime to $N^{-}p$ Tate module $\varprojlim_{(M,N^{-}p)=1} A[M]$. 
The notation $\eta^{(p)}$ denotes a full level structure of level prime to $N^{-}p$ given by an 
$\cO_{B}$-linear isomorphism 
$$\eta^{(p)}: \cO_{B} \otimes_\Z \widehat{\Z}^{(N^{-}p)} \iso \cT^{(N^{-}p)}(A)$$
for $\widehat{\Z}^{(N^{-}p)}= \prod_{l\ndivide N^{-}p} \Z_l$.\\
%(PM5) Let $\Lie_{S}(A)$ be the relative Lie algebra of $A$. 
%There exists an $ O \otimes_{\Z} \cO_S$-module isomorphism $\Lie_{S}(A) \iso O \otimes_{\Z} \cO_S$, locally under the Zariski topology of $S$.\\
\\
The notation $\sim$ denotes up to a prime to $N^{-}p$ isogeny.\\
\begin{thm}[Morita, Kottwitz]
 The functor $\cF^{(p)}$ is represented by a smooth proper pro-scheme $Sh^{(p)}_{/\Z[\frac{1}{N}]}$. 
%Moreover, 
%there exists an isomorphism given by
%$$ Sh^{(p)} \times \Q \iso Sh/G(\Z_p)_{/ \Q}.$$
(\cf \cite[\S 2.2]{Br}).\\
\end{thm}
\noindent \\
Let $\cA$ be the universal abelian surface.\\
\\
\\
\subsection{Idempotent} 
In this subsection, we describe generalities regarding an idempotent in the ring of algebraic correspondences on the universal abelian surface.\\
%which play a key role in the rest of the article.\\
\\
Let the notation and hypotheses be as in \S2.1. 
Let $\epsilon \in \cO_{B} \otimes \cO_{M_{0}} [\frac{1}{2p_{0}}]$ be the non-trivial idempotent given by
$$
\epsilon = \frac{1}{2}\bigg{(}  1\otimes 1 + \frac{1}{p_{0}} j \otimes \sqrt{p_{0}}\bigg{)}.
$$
\noindent\\
%We have an isommorphism
%$$  \iota_{M_{0}}: B \otimes M_{0} \iso M_{2}(M_{0}).                  $$
Based on the hypotheses (H1) and (H2), we have an optimal embedding 
$$ \iota_{K}: K \hookrightarrow B$$
(\cf \cite[\S2.4]{Br}).\\
\\
In view of the moduli interpretation of the Shimura variety $Sh^{(p)}$, we can let $\epsilon$ 
naturally act on the test objects as a correspondence. 
It follows that $\epsilon$ can be regarded as an algebraic correspondence on the universal abelian surface $\cA$.\\
\\
Via the $p$-adic embedding $\iota_p$, we often regard the idempotent $\epsilon$ as an element in $\cO_{B} \otimes \Z_{p}$. 
In view of the moduli interpretation, the $p$-divisible group $\cA[p^{m}]$ has a natural structure of an $\cO_{B} \otimes \Z_{p}$-module 
for a positive integer $m$. 
We thus have a natural action of the idempotent $\epsilon$ on the $p$-divisible group $\cA[p^{m}]$. 
The action plays a key role in the article.\\
\\
Heuristically, the idempotent often reduces the complexity arising from the non-commutative nature of the quaternion algebra to a commutative one. 
Such situations occasionally arise in the article.\\
\begin{remark}
An analogous idempotent arises in the case of an unitary model of a quaternionic Shimura variety over a totally real field (\cf \cite{Ca} and \cite{Bu1}).
\end{remark}
\noindent\\
\subsection{CM points} In this subsection, we introduce notation regarding CM points on the Shimura variety.\\
\\
Let the notation and hypotheses be as in \S2.2. 
Let $b$ be a positive integer prime to $N$ and recall that $\cO_{b}$ denote the order $\Z+b\cO$. 
We also recall that $\Pic(\cO_{b})$ denotes the corresponding ring class group 
of conductor $b$.\\
\\
Associated to an ideal class $[\mathfrak{a}] \in \Pic(\cO_{b})$, we have the corresponding CM point 
$x(\mathfrak{a})$ on the Shimura variety $Sh^{(p)}$ (\cf \cite[\S2.4]{Br} and \cite[\S1.1]{LZZ}). 
Moreover, these CM points are Galois conjugates (\cf \cite[\S2.4]{Br}).\\
\\
\\
\subsection{Igusa tower}
\noindent In this subsection, we briefly recall the notion of 
$p$-ordinary Igusa tower over the Shimura variety. 
We follow \cite[Ch. 8]{Hi1} and \cite[\S1]{Hi2}.\\
\\
%Let $\overline{\Q}$ be an algebraic closure of $\Q$ and $\overline{\Q}_p$ be an algebraic closure of $\Q_p$. 
%We fix a complex embedding $\iota_\infty: \overline{\Q} \hookrightarrow \C$ and a $p$-adic embedding 
%$\iota_p: \overline{\Q} \hookrightarrow \overline{\Q}_p$.\\
%\\
Let the notation and hypotheses be as in \S2.2. 
Recall that $p$ is a prime such that $p \ndivide N$. 
Let $\cW$ be the strict Henselisation inside $\overline{\Q}$ of the local ring of $\Z_{(p)}$ corresponding to $\iota_p$. 
Let $\F$ be the residue field of $\cW$. Note that $\F$ is an algebraic closure of $\F_p$.\\
\\
Let $Sh^{(p)}_{/\cW} = Sh^{(p)} \times_{\Z[\frac{1}{N}]} \cW$ and $Sh^{(p)}_{/ \F}= Sh^{(p)}_{/\cW} \times_{\cW} \F$.\\
\\
From now, let $Sh$ denote $Sh^{(p)}_{/ \F}$. 
Recall that $\cA$ denotes the universal abelian surface over $Sh$.\\
\\
Let $Sh^{ord}$ be the subscheme of $Sh$ i.e. the locus on which the Hasse-invariant does not vanish. It is an open dense subscheme. 
Over $Sh^{ord}$, the connected part $\cA[p^m]^{\circ}$ of $\cA[p^m]$ is étale-locally isomorphic to 
$\mu_{p^{m}} \otimes_{\Z_p} \cO_{B}^\circ$ as an $\cO_{B,p}$-module. Here $\cO_{B,p}=\cO_{B} \otimes \Z_{p}$ and $\cO_{B}^\circ$ is an $\cO_{B,p}$-direct summand of $\cO_{B,p}$. 
Moreover, $\cO_{B}^\circ$ is a maximally isotropic subspace of $\cO_{B,p}$ with respect to the bilinear form associated to the quadratic form arising from the reduced norm 
such that $\rank_{\Z_p} \cO_{B}^\circ = 2$. 
Let $\cO_{B}^{\acute{e}t}$ be the orthogonal complement of $\cO_{B}^\circ$. 
For an explicit description of both, we refer to the end of \S3.1.\\
\\
We now define the Igusa tower. For $m \in \mathbb{N}$, the $m^{th}$-layer of the Igusa tower over $Sh^{ord}$ is defined by
\beq Ig_m=\Isom_{\cO_{B,p}}(\mu_{p^m}\otimes_{\Z_p} \cO_{B}^{\circ}, \cA[p^m]^\circ).\eeq
Note that the projection $\pi_m: Ig_m \rightarrow Sh^{ord}$ is finite and étale.\\
\\
In view of the description of the moduli functor $\cF^{(p)}$, we have the following isomorphism
\beq Ig_m\iso \Isom_{\epsilon \cO_{B,p}}(\mu_{p^m}\otimes_{\Z_p} \epsilon \cO_{B}^{\circ}, \epsilon(\cA[p^m]^\circ)).\eeq 
The Cartier duality  and the polarisation $\bar{\lam}_x$ induces an isomorphism
\beq Ig_m\iso \Isom_{\cO_{B,p}} (\cO_{B}^{\acute{e}t}/p^{m}\cO_{B}^{\acute{e}t}, \cA[p^m]^{\acute{e}t}) \iso \Isom_{\epsilon \cO_{B,p}} (\epsilon(\cO_{B}^{\acute{e}t}/p^{m}\cO_{B}^{\acute{e}t}), \epsilon(\cA[p^m]^{\acute{e}t})).\eeq
\noindent\\
The full Igusa tower over $Sh^{ord}$ is defined by
\beq Ig=\varprojlim Ig_m = \Isom_{\cO_{B,p}}(\mu_{p^\infty}\otimes_{\Z_p} \cO_B^{\circ}, \cA[p^\infty]^\circ)\iso \Isom_{\epsilon \cO_{B,p}}(\mu_{p^\infty}\otimes_{\Z_p} \epsilon \cO_B^{\circ}, \epsilon(\cA[p^\infty]^\circ)).\eeq
In view of (2.5) and (2.6), it suffices to study the projection by $\epsilon$ of the $p$-divisible group.\\
\\
(Et) Note that the projection $\pi: Ig \rightarrow Sh^{ord}$ is étale.\\
\\
Let $x$ be a closed ordinary point in $Sh$ and $A_x$ the corresponding abelian surface. 
We have the following description of the $p^\infty$-level structure on $A_x[p^\infty]$.\\
\\
(PL) Let $\eta_{p}^\circ$ be a $p^\infty$-level structure on $A_x[p^\infty]^\circ$. 
It is 
given by an $\cO_{B,p}$-module isomorphisms $\eta_{p}^\circ: \mu_{p^\infty}\otimes_{\Z_p} \cO_B^{\circ} \iso A_x[p^\infty]^\circ$. 
The Cartier duality  and the polarisation $\bar{\lam}_x$ induces an isomorphism 
$\eta_{p}^{\acute{e}t}: O_{B,p}^{\acute{e}t} \iso A_x[p^\infty]^{\acute{e}t}$.\\
\\
Let V be an irreducible component of $Sh$ and $V^{ord}$ the intersection $V \cap Sh^{ord}$. 
Let $I$ be the inverse image of $V^{ord}$ under $\pi$. In \cite[Ch.8]{Hi1} and \cite{Hi2}, it has been shown that\\
\\
(Ir) $I$ is an irreducible component of $Ig$.\\
\\
\\
\subsection{Mod $p$ modular forms}
In this subsection, we briefly recall the notion of mod $p$ modular forms on an irreducible component of the 
Shimura variety.\\
\\
Let $V$ and $I$ be as in \S 2.5. Let $C$ be an $\F$-algebra. The space of mod $p$ modular forms on $V$ over $C$ is defined by
\beq M(V,C)=H^0(I_{/C},\cO_{I_{/C}})\eeq
for $I_{/C}:= I \times_{\F} C$.\\
\\
In view of \S2.2 and \S2.5, we have the following geometric interpretation of mod $p$ modular forms.\\
\\
A mod $p$ modular form is a function $f$ of isomorphism classes of $\tilde{x}=(x,\eta_p^{\circ})_{/C'}$ where $C'$ is a $C$-algebra, 
$x=(A,\iota,\bar{\lam},\eta^{(p)})_{/C'} \in \cF^{(p)}(C')$ and $\eta_p^{\circ}:\mu_{p^\infty}\otimes_{\Z_p} \cO_{B}^\circ \iso A[p^\infty]^\circ$ is 
an $\cO_{B,p}$-linear isomorphism, such that the following conditions are satisfied.\\
\\
(G$1$) $f(\tilde{x}) \in C'$.\\
(G$2$) If $\tilde{x} \iso \tilde{x}'$, then $f(\tilde{x})=f(\tilde{x}')$. 
Here $\tilde{x} \iso \tilde{x}'$ denotes that $x \iso x'$ and the corresponding isomorphism between $A$ and $A'$ induces an isomorphism between $\eta_p^{\circ}$ and $\eta_p'^{\circ}$.\\
(G$3$) $f(\tilde{x}\times_{C'} C'')=h(f(\tilde{x}))$ for any $C$-algebra homomorphism $h:C' \rightarrow C''$.\\
\\
We have an analogous notion of $p$-adic (resp. classical) modular forms (resp., \cf \cite[Ch. 8]{Hi1} and \cite[\S3.1]{Br}). 
In the article, we often regard classical modular forms as $p$-adic modular forms in the usual way.\\
\\
\\
\section{Serre-Tate deformation space} 
\noindent In this section, we describe certain aspects of the Serre-Tate deformation space of the Shimura variety. 
In \S3.1-3.2, we describe generalities regarding the Serre-Tate deformation theory of an ordinary closed point on the Shimura variety. 
In \S3.3, we consider certain Hecke operators and determine their action on the Serre-Tate expansion of classical modular forms around CM points 
on the Shimura variety.\\
\\
The deformation theory plays a foundational role in Chai's theory of Hecke stable subvarieties of a Shimura variety. 
This section thus plays a key role in the article.\\
\\
\subsection{Serre-Tate deformation theory}
In this subsection, we briefly recall Serre-Tate deformation theory of an ordinary closed point on the Shimura variety. We follow \cite[\S8.2]{Hi1}, \cite[\S2]{Hi3} and 
\cite[\S1]{Hi4}.\\
\\
Let the notation and assumptions be as in \S2.5. 
In particular, $Sh$ denotes the Shimura variety of level prime to $N^{-}p$ and 
$Sh^{ord}$ the $p$-ordinary locus. 
Let $x$ be a closed point in $Sh^{ord}$ carrying $(A_x,\iota_x,\bar{\lam}_x,\eta_{x}^{(p)})_{/\F}$. Let $V$ be the irreducible component of 
$Sh$ containing $x$.\\
\\
Let $CL_W$ be the category of complete local $W$-algebras with residue field $\F$. 
Let $\cD_{/W}$ be the fiber category over $CL_W$ of deformations of $x_{/\F}$ defined as follows. 
Let $R \in CL_W$. The objects of $\cD_{/W}$ over $R$ consist of $x'^{*}=(x',\iota_{x'})$, where $x' \in \cF^{(p)}(R)$ and 
$\iota_{x'}: x' \times_R \F \iso x$. Let $x'^{*}$ and $x''^{*}$ be in $\cD_{/W}$ over $R$. By definition, a morphism $\phi$ between 
$x'^{*}$ and $x''^{*}$ is a morphism (still denoted by) $\phi$ between $x'$ and $x''$ satisfying \cite[(7.3)]{Hi1} and the following condition.\\
\\
(M) Let $\phi_0$ be the special fiber of $\phi$. The automorphism $\iota_{x''} \circ \phi_0 \circ \iota_{x'}^{-1}$ of $x$ equals the identity.\\
\\
Let $\widehat{\cF}_{x}$ be the deformation functor given by\\
$$\widehat{\cF}_{x}: CL_{/W} \rightarrow SETS$$
\beq R \mapsto \{ x'^{*}_{/R} \in \cD\}/\iso .\eeq
\\
The notation $\iso$ denotes up to an isomorphism.\\
\\
As $R\in CL_W$, by definition $R$ is a projective limit of local $W$-algebras with nilpotent maximal ideal. 
We can (and do) suppose that $R$ is a local Artinian $W$-algebra with the nilpotent maximal ideal $m_R$. 
Let $x'^{*}_{/R} \in \cD$ and $A$ denote $A_{x'}$. By Drinfeld's theorem (\cf \cite[\S8.2.1]{Hi1}), $A[p^\infty]^\circ(R)$ is killed 
by $p^{n_0}$ for sufficiently large $n_0$. Let $y \in A(\F)$ and $\tilde{y}\in A(R)$ such that $\tilde{y}_0=y$ 
for $\tilde{y}_0$ being the special fiber of $\tilde{y}$. 
As $A_{/R}$ is smooth, such a lift always exists. 
By definition, $\tilde{y}$ is determined modulo $\ker(A(R)\mapsto A(\F))=A[p^\infty]^\circ(R)$. 
It thus follows that for $n\geq n_0$, $``p^n"y_0:=p^n\tilde{y}$ is well defined. 
From now, we suppose that $n\geq n_0$. If $y \in  A[p^n](\F)$, 
then $``p^n"y \in A[p^{\infty}]^{\circ}(R)$ by definition. 
We now consider the $\epsilon$-components (\cf \S2.3). 
Recall that it suffices to consider the $\epsilon$-components in view of the isomorphisms (2.5) and (2.6).\\
\\
We thus have a homomorphism
\beq ``p^n" : \epsilon A[p^n](\F) \rightarrow  \epsilon A[p^{\infty}]^{\circ}(R).\eeq
We also have the commutative diagram
\begin{diagram}
 \epsilon A[p^{n+1}]^{\acute{e}t}(R) &\rTo^{\iso} & \epsilon A[p^{n+1}](\F) &\rTo^{``p^{n+1}"} & \epsilon A[p^{\infty}]^{\circ}(R)\\
\dTo_{p} & &\dTo_{p} & &\dTo_{=}\\
\epsilon A[p^{n}]^{\acute{e}t}(R) &\rTo^{\iso} & \epsilon A[p^{n}](\F) &\rTo^{``p^{n}"} & \epsilon A[p^{\infty}]^{\circ}(R).
\end{diagram}
\noindent\\
Passing to the projective limit, this gives rise to a homomorphism
\beq ``p^\infty" : \epsilon A[p^\infty](\F) \rightarrow \epsilon A[p^{\infty}]^{\circ}(R).\eeq
\\
(CC) Let $y=\varprojlim{y_n} \in  \epsilon A[p^\infty](\F) \iso  \epsilon A_{x}[p^\infty]^{\acute{e}t}$, 
where $y_n \in \epsilon A[p^n](\F)$ and the later isomorphism is induced by $\iota_{x'}$. 
Let $q_{n,p}(y_n)=``p^n"y_n$ and $q_{p}(y) = \lim q_{n,p}(y_n)$. 
By definition, 
$$q_{p}(y) \in \epsilon A[p^{\infty}]^{\circ}(R)\iso 
\Hom(\epsilon ^{t}A_{x}[p^{\infty}]^{\acute{e}t}, \widehat{\mathbb{G}}_{m}(R)).$$ 
Let $q_{A,p}$ be the pairing given by\\
$$q_{A,p}:\epsilon A_{x}[p^{\infty}]^{\acute{e}t} \times \epsilon ^{t}A_{x}[p^{\infty}]^{\acute{e}t} \rightarrow \widehat{\mathbb{G}}_{m}(R)$$
\beq q_{A,p}(y,z)=q_{p}(y)(z).\eeq
\\
We have the following fundamental result.\\
\begin{thm}[Serre-Tate] 
Let the notation be as above.\\
 (1). There exists a canonical isomorphism
\beq \widehat{\cF}_{x}(R) \iso 
\Hom_{\Z_p}( \epsilon A_x[p^\infty ]^{\acute{e}t} \times \epsilon^{t}{}A_x[p^\infty ]^{\acute{e}t}, \widehat{\mathbb{G}}_{m}(R)) \eeq
given by $x'^{*} \mapsto q_{A_{x'},p}$.\\
(2). The deformation functor $\widehat{\cF}_{x}$ is represented by the formal scheme $\widehat{S}_{/W}:=\Spf(\widehat{\cO}_{V,x})$. 
A $p^\infty$-level structure as in (PL), gives rise to a canonical isomorphism of the deformation space $\widehat{S}_{/W}$ with 
the formal torus $\widehat{\mathbb{G}}_{m}$ (\cf \cite[Prop. 1.2]{Hi4}).\\
\end{thm}
\begin{proof} We first prove (1). 
It follows from an argument similar to the proof of \cite[Thm. 8.9]{Hi1}. 
We sketch the details for convenience.\\
\\
A fundamental result of Serre and Tate states that deforming $x$ over $CL_W$ is 
equivalent to deform the corresponding Barsotti-Tate $\cO_{B,p}$-module $A_x[p^\infty]_{/\F}$ over $CL_W$ (\cf \cite[\S1]{Ka1} and \cite[\S8.2.2]{Hi1}). 
In view of the isomorphisms (2.5) and (2.6), 
deforming the Barsotti-Tate $\cO_{B,p}$-module $A_x[p^\infty]_{/\F}$ is equivalent to deform 
the Barsotti-Tate $\epsilon \cO_{B,p}$-module $\epsilon A_x[p^\infty]_{/\F}$.\\
%For a related argument, we refer to \cite[\S8.2.5]{Hi1}. 
\\
Let $R$ be as above. 
Consider the sheafification $\underline{\Hom}_{R_{fppf}}(\epsilon A[p^{n}]^{\circ}, \epsilon A[p^{n}]^{\acute{e}t})$ (resp. $\underline{Ext^{1}}_{R_{fppf}}(\epsilon A[p^{n}]^{\circ}, \epsilon A[p^{n}]^{\acute{e}t})$) 
of the presheaf $U \mapsto \Hom_{U}(\epsilon A[p^{n}]^{\circ}_{/U}, \epsilon A[p^{n}]^{\acute{e}t}_{/U})$ (resp. $U \mapsto Ext^{1}_{U}(\epsilon A[p^{n}]^{\circ}_{/U}, \epsilon A[p^{n}]^{\acute{e}t}_{/U})$). 
Recall that a connected-étale extension $\epsilon A[p^{n}]^{\circ} \rightarrow X \twoheadrightarrow \epsilon A[p^{n}]^{\acute{e}t}$ 
in the category of finite flat $\Z/p^{n}\Z$-modules over $R$ splits over an {\it{fppf}}-extension $R'/R$. 
We thus have 
$$\underline{Ext^{1}}_{R_{fppf}}(\epsilon A[p^{n}]^{\circ}, \epsilon A[p^{n}]^{\acute{e}t})=0$$ 
and a splitting 
$$X=\epsilon A[p^{n}]^{\circ} \oplus \epsilon A[p^{n}]^{\acute{e}t}.$$ 
Choice of a section in the splitting gives rise to a homomorphism $\phi_{R'}\in \Hom_{R'}(\epsilon A[p^{n}]^{\circ}_{/R'}, \epsilon A[p^{n}]^{\acute{e}t}_{/R'})$. 
By construction, $R' \mapsto \phi_{R'}$ satisfies the descent datum. 
We thus have a morphism 
$$Ext^{1}_{R_{fppf}}(\epsilon A[p^{n}]^{\circ}, \epsilon A[p^{n}]^{\acute{e}t}) \rightarrow H^{1}(R_{fppf}, \underline{\Hom}_{R_{fppf}}(\epsilon A[p^{n}]^{\circ}, \epsilon A[p^{n}]^{\acute{e}t})).$$
By an {\it{fppf}}-descent, this is in fact an isomorphism. 
We conclude that 
$$
Ext^{1}_{R_{fppf}}(\epsilon A[p^{\infty}]^{\circ}, \epsilon A[p^{\infty}]^{\acute{e}t}) 
\iso \varprojlim_{n} Ext^{1}_{R_{fppf}}(\epsilon A[p^{n}]^{\circ}, \epsilon A[p^{n}]^{\acute{e}t})
\iso \Hom_{\Z_p}( \epsilon A_x[p^\infty ]^{\acute{e}t} \times \epsilon^{t}{}A_x[p^\infty ]^{\acute{e}t}, \widehat{\mathbb{G}}_{m}(R)).
$$
The last isomorphism follows from the additional fact that 
$ \widehat{\mathbb{G}}_{m} \iso \varprojlim_{n} R^{1}\pi_{*}\mu_{p^n} \iso  \widehat{\mathbb{G}}_{m}$. 
Here $\pi: R_{fppf} \rightarrow R_{\acute{e}t}$ is the projection for the small étale site $R_{\acute{e}t}$.\\
\\
%Recall, $\epsilon A_x[p^\infty]_{/\F} = \prod_{\mathfrak{p} \in \Sigma_p} \epsilon A_x[\mathfrak{p}^\infty]_{/\F}$. Thus, the deformation functor $\widehat{\cF}_{x}$ has a natural decomposition\\
%\beq \widehat{\cF}_{x} = \prod_{\mathfrak{p}\in \Sigma_{p}} \widehat{\cF}_{x,\mathfrak{p}} ,\eeq
%\\
%where $\widehat{\cF}_{x,\mathfrak{p}}$ is the deformation functor of the Barsotti-Tate $\epsilon O_D$-module $\epsilon A_x[\mathfrak{p}^\infty]_{/\F}$. 
%By an argument similar to the one in [loc. cit.], 
We conclude that the deformation functor $\widehat{\cF}_{x}$ is represented by 
$\Hom_{\Z_p}(\epsilon A_x[p^\infty ]^{\acute{e}t} \times \epsilon ^{t}A_x[p^\infty ]^{\acute{e}t}, \widehat{\mathbb{G}}_{m}(R))$ and 
%Now, in view of (PL), for $\mathfrak{p}\in \Sigma_{r,p}$ the Barsotti-Tate $\epsilon O_D$-module $\epsilon A_x[\mathfrak{p}^\infty]_{/\F}$ does not deform. 
this finishes the proof of (1).\\
\\
We now prove (2). The first part of (2) follows from general deformation theory. The polarisation $\lam_x$ induces a canonical $\cO_B$-linear
isomorphism $A_x \iso ^{t}A_x$. 
%Note that
%$\epsilon A_x[p^\infty ]^{\acute{e}t}=  \epsilon A_x[p^\infty ]^{\acute{e}t} \times \epsilon A_x[\mathfrak{\overline{\mathfrak{P}}}^\infty ]^{\acute{e}t}$
%and 
%$\Hom_{\Z_p}(\epsilon A_x[\mathfrak{p}^\infty ]^{\acute{e}t} \times \epsilon A_x[\mathfrak{p}^\infty ]^{\acute{e}t}, \widehat{\mathbb{G}}_{m})= \Hom_{\Z_p}(\epsilon A_x[\mathfrak{P}^\infty ]^{\acute{e}t} \times \epsilon %A_x[\mathfrak{\overline{P}}^\infty ]^{\acute{e}t}, \widehat{\mathbb{G}}_{m})$. 
For a given choice of $p^\infty$-level structure
$\eta_{p}^{\acute{e}t}$ as in (PL), 
we have a canonical isomorphism 
$\Hom_{\Z_p}(\epsilon A_x[p^\infty ]^{\acute{e}t} \times \epsilon A_x[p^\infty ]^{\acute{e}t}, \widehat{\mathbb{G}}_{m}) \iso \widehat{\mathbb{G}}_{m}$. 
This finishes the proof.\\
\end{proof}
\noindent\\
\noindent Let $x_{ST}$ be the universal deformation. A $p^{\infty}$-level structure $\eta_{p}^{\circ}$ of $x$ gives rise to a canonical 
$p^{\infty}$-level structure $\eta_{p,ST}^{\circ}$ of the universal deformation $x_{ST}$. \\
\\
We now recall a few definitions.\\
\begin{defn}Recall that a $p^\infty$-level structure as in (PL), gives rise to a canonical isomorphism of the deformation space 
$\widehat{S}_{/W}$ with the formal 
torus $\widehat{\mathbb{G}}_{m} $ (\cf part (2) of Theorem 3.1). 
Under this identification, let $t$ be the co-ordinate of the deformation space 
$\widehat{S}_{/W}$ for $t$ being the usual co-ordinate of $\widehat{\mathbb{G}}_{m}$. 
We call $t$ the Serre-Tate co-ordinate of the deformation space $\widehat{S}_{/W}$.\\
\end{defn}
\noindent
\begin{defn} The deformation corresponding to the identity element in the deformation space is said to be the canonical lift of $x$.
\end{defn}
\noindent\\
We have $\widehat{S}=\Spf(\widehat{W[X(S)]})$ for $S=\mathbb{G}_{m}$ 
and $\widehat{W[X(S)]}$ being the completion along the augmentation ideal. Here $X(S)$ is the character group of $S$. 
Note that $W[X(S)]$ is the ring consisting of formal finite sums $\sum_{\xi \in \Z} a(\xi) t^{\xi}$ for $a({\xi}) \in W$ and $t$ being the 
usual co-ordinate of $\mathbb{G}_{m}$.\\
%Here, $t^{\xi}$ is the character given by $t^{\xi} ( t \otimes u)= t^{\Tr_{F/\Q}(\xi u)}$, for $u \in O^{*}$.\\
\begin{defn}
Let $f$ be a mod $p$ modular form over $\F$ (\cf \S2.6). 
We call the evaluation $f((x_{ST},\epsilon \eta_{p,ST}^{\circ})) \in \widehat{\F[X(S)]}$ 
as the $t$-expansion of $f$ around $x$.\\
\end{defn}
\noindent\\We have the following $t$-expansion principle.\\
\\
\noindent ($t$-expansion principle) The $t$-expansion of $f$ around $x$ determines $f$ uniquely (\cf (Ir)).\\
\\
We have an analogous $t$-expansion principle for $p$-adic modular forms (\cf \cite[Ch. 8]{Hi1}).\\
\\
We end this subsection with a description of the modules $\cO_{B}^\circ$ and $\cO_{B}^{\acute{e}t}$ (\cf \S2.5).\\
\\
We follow \cite[pp. 9-10]{Hi2}. We fix an identification of $\cO_{B,p}$ with $\Z_{p}^4=(\Z_{p}^{2})^{2}$. 
Let $\cO_{B,p}$ act on itself by the left multiplication on $\Z_{p}^2$. Let $x$ be a closed point in $Sh^{ord}$ carrying $(A_x,\iota_x,\bar{\lam}_x,\eta_{x}^{(p)})_{/\F}$. 
Let $\tilde{x}=(A_{x},\iota_x,\bar{\lam}_x,\eta_{x}^{(p)})_{/W}$ be the canonical lift of $x$. 
As $A_{x/W}$ has CM, it descends to $A_{x/\cW}$ (cf. \cite{Sh3}). 
Let $A_{x/\C}=A_x \times_{\cW,\iota_\infty} \C$. 
Note that $\cT(A_{x})_{/\overline{E}}\iso H_{1}(A_x(\C),\A_{\Q,f}) = H_{1}(A_x(\C),\Z) \otimes_\Z \A_{\Q,f}$. 
Let $\cT_{p}A_{x/\C}$ be the $p$-adic Tate-module of $A_{x/\C}$. We choose an identification of $\cT_{p}A_{x/\C}$ 
with $\cO_{B,p}=H_{1}(A_x(\C),\Z_p)$. The connected-étale exact sequence $A_{x}[p^\infty]^\circ \hookrightarrow A_x[p^\infty] \twoheadrightarrow A_x[p^\infty]^{\acute{e}t}$ of 
the $p$-divisible group $A_x[p^\infty]_{/\F}$ 
induces an exact sequence $\cO_B^\circ \hookrightarrow \cO_{B,p} \twoheadrightarrow \cO_B^{\acute{e}t}$ of $\cO_{B,p}$-modules. As $A_{x/\cW}$ is 
the canonical lift, it follows that $\cO_{B,p}$ splits as $\cO_B^\circ \oplus \cO_B^{\acute{e}t}$ such that the lift of the Frobenius has unit eigenvalues on $\cO_B^{\acute{e}t}$ and non-unit eigenvalues 
on $\cO_B^\circ$.\\
\\
\subsection{Reciprocity law} In this subsection, we describe the action of the local algebraic stabiliser of 
an ordinary closed point on 
the Serre-Tate co-ordinates of the deformation space. 
This can be considered as an infinitesimal analogue of Shimura's reciprocity law.\\
\\
Let the notation and hypotheses be as in \S3.1.
In particular, $Sh$ denotes the Shimura variety of level prime to $N^{-}p$. 
Let  $g\in G(\Z_{(p)})$  acts on  $Sh$  through the right multiplication on the prime to $N^{-}p$ level structure i.e., 
$$ (A,\iota,\bar{\lam},\eta^{(p)})_{/S} \mapsto  (A,\iota,\bar{\lam},\eta^{(p)}\circ g)_{/S}$$
(\cf \S2.2).\\
\\
Recall that $x$ is a closed ordinary point in $Sh$ with a $p^\infty$-level structure $\eta_{p}^{ord}$. 
%Let $(K_{x},\Sigma_{x})$ be the CM-type of $x$. 
We suppose that $x$ has CM by $K$ and the existence of an embedding
$\iota_{x}:\cO \hookrightarrow \End(A)$.
Let $H_{x}(\Z_{(p)})$ be the stabiliser of 
$x$ in $G(\Z_{(p)})$. It is explicitly given by  
\beq
H_{x}(\Z_{(p)})=(\Res_{\cO_{(p)}/\Z_{(p)}}\mathbb{G}_m)(\Z_{(p)})=\cO_{(p)}^{\times} 
\eeq
for $\cO_{(p)}=\cO \otimes \Z_{(p)}$ (\cf \cite[\S3.2]{Hi3} and \cite{Sh3}). 
We call $H_{x}(\Z_{(p)})$ the local algebraic stabiliser of $x$.\\
\\
Recall that $c$ denotes the complex conjugation of $K$.\\
\\
Let $\alpha \in H_{x}(\Z_{(p)})$. 
%Recall, $\alpha$ induces an endomorphism $\iota_x(\alpha)$ of $A_x$. 
%It acts on $\epsilon A_x$ as $\epsilon \alpha$. By the abuse of notation, we still denote this by $\alpha$. In this way, we regard 
%$\alpha \in \prod_{\mathfrak{p}\in\Sigma_p} O_{(\mathfrak{p})}$, where $O_{(\mathfrak{p})}$ is the localisation of $O$ at the prime ideal $\mathfrak{p}$. 
%For $\mathfrak{p}\in \Sigma_p$, $\alpha$ acts on $\epsilon A_x[\mathfrak{p}^\infty]$ via its $\mathfrak{p}$-component $\alpha_{\mathfrak{p}}$. 
Let $R \in CL_W$, $x'\in \widehat{\cF}_{x}(R)$ and $A=A_{x'}$. We have the following \\
\begin{diagram}
\Hom(\epsilon A_x[p^{\infty}]^{\acute{e}t},\hat{\mathbb G}_{m}(R)) &\rInto &\epsilon A[p^{n}] &\rTo &\epsilon A_x[p^{n}]^{\acute{e}t}(R)\\
\uTo_{\alpha} & &\dTo_{\id} & &\dTo_{\alpha^{-c}}\\
\Hom(\epsilon A_x[p^{\infty}]^{\acute{e}t},\hat{\mathbb G}_{m}(R)) &\rInto &\epsilon A[p^{n}] &\rTo &\epsilon A_x[p^{n}]^{\acute{e}t}(R)\\
\end{diagram}
\\
commutative diagram.\\
\\
Let $u=\varprojlim u_{n} \in \epsilon A_x[p^{\infty}]^{\acute{e}t}$ and 
$q_{p}(u) \in \epsilon A_x[p^{\infty}]^{\circ}(R) \iso \Hom(\epsilon A_x[p^{\infty}]^{\acute{e}t},\hat{\mathbb G}_{m}(R))$ be as in (CC). 
We recall that the last isomorphism is induced by the Cartier duality composed with the polarisation $\lam_{x}$. Thus, if $\alpha$ is prime to $p$, then it acts on 
$q_{p}(u)$ by 
$$q_{p}(u) \mapsto \varprojlim \alpha(``p^{n}"\alpha^{-c_{x}}(u_n))=q_{p}(\alpha)^{\alpha^{1-c_{x}}}.$$
In other words, $\alpha$ acts on $q_{p}$ by $q_{p} \mapsto q_{p}^{\alpha^{1-c}}$. 
Fix a $p^\infty$-level structure as in (PL). This gives rise to the Serre-Tate co-ordinate 
$t$ of the deformation space $\widehat{S}_{/W}$ (cf. (CC)).\\
\\
In view of the above discussion and Theorem 3.1, we have the following reciprocity law.\\
\begin{lm} 
Let the notation and the assumptions be as above. 
If $\alpha \in H_{x}(\Z_{(p)})$ is prime to $p$, then it acts on the Serre-Tate 
co-ordinate $t$ by $t\mapsto t^{\alpha^{1-c}}$.\\
\end{lm}
\noindent\\
The above simple lemma plays a key role in the proof of a linear independence of mod $p$ modular forms (\cf \S4).\\
\\
\\
\subsection{Hecke operators} In this subsection, we describe certain Hecke operators on the space of classical modular forms. 
In the case of modular curves, such operators are considered in \cite[\S1.3.5]{Hi5}.\\
\\
Let the notation and hypotheses be as in \S3.1.
Let $(\zeta_{p^{n}}=\exp(2 \pi i/ p^n))_{n} \in \overline{\Q}$ be a compatible system of $p$-power roots of unity. Via $\iota_{p}$, 
we regard it as a compatible system in $\C_p$. Let $W_{n}=W[\mu_{p^{n}}]$.\\
\\
Let $g$ be a classical modular form on a quotient Shimura curve $Sh_{K}$ of the Shimura variety $Sh$ over $\cW$ (\cf \S2.1). 
Via $\iota_{\infty}$, we regard it as a modular form over $\C$. 
Let $z$ denote the complex variable of the complex Shimura curve or that of $\mathfrak{H}$ 
for $\mathfrak{H}$ being the upper half complex plane.\\
\\
Let $\phi:\Z/p^{r}\Z \rightarrow \cW$ be an arbitrary function with the normalised Fourier transform $\phi^{*}$ given by
\beq
\phi^{*}(y)=\frac{1}{p^{r/2}} \sum_{u \in \Z/p^{r}\Z} \phi(u)\exp(yu/p^{r}) 
\eeq
for $y \in \Z/p^{r}\Z$.\\
\\
Let $g|\phi$ be the classical modular form given by
\beq g|\phi(z)= \sum_{u \in \Z/p^{r}\Z} \phi^{*}(-u) g(z+u/p^r).
\eeq
\noindent\\
\\
We now regard the above classical modular forms as $p$-adic.
 The action of the Hecke operator on the $t$-expansion around a CM point (\cf Definition 3.4) is the following.\\
\begin{prop} 
Let the notation be as above. 
Let $g(t)=\sum_{\xi \in \Z_{\geq 0}} a(\xi,g) t^{\xi}$ be the $t$-expansion of $g$ around $x$.
We have 
$$g|\phi(t)= \sum_{\xi \in \Z_{\geq 0}} \phi(\xi)a(\xi,g) t^{\xi}.$$
\end{prop}
\begin{proof}
Let $u\in \Z$ and $\alpha(u/p^{r})\in G(\A_{\Q,f})$ such that 
$$\alpha(u/p^{r})_{p}=\MX{1}{u/p^{m}}{0}{1}$$ 
and $\alpha(u/p^{r})_{l}=1$, for $l \neq p$.\\
\\
We have
$$ g|\phi= \sum_{u \in \Z/p^{r}\Z} \phi^{*}(-u) g|\alpha(u/p^{r}),
$$
as $\alpha(u/p^{r})$ acts on $\mathfrak{H}$ by $z \mapsto z+u/p^{r}$. 
Here $\mathfrak{H}$ denotes the complex upper half plane as before.\\
\\
The isogeny action of  $\alpha(u/p^{r})$ on the Igusa tower $\pi: Ig \rightarrow Sh_{/W_{r}}$ preserves the deformation space 
$\widehat{S}$ and induces $t\mapsto \zeta_{p^{r}}t ^{u}$ (\cf \cite[Lem. 4.14]{Br}).\\
\\
In view of the Fourier inversion formula, this finishes the proof.\\
\end{proof}
\noindent 
\noindent 
The above Hecke operators naturally arise during the determination of the $\mu$-invariant of 
a class of anticyclotomic Rankin-Selberg $p$-adic L-functions (\cf Lemma 5.4 and proof of Theorem 5.7). 
The above proposition accordingly helpful in the determination of the $\mu$-invariants.\\
\\
\\
\section{Linear independence}
\noindent In this section, we consider a linear independence of mod $p$ modular forms based on Chai's theory of Hecke stable subvarieties of a Shimura variety. 
In \S4.1, we describe the formulation. In \S4.2-4.6, we prove the independence.\\
\\
\\
\subsection{Formulation} In this subsection, we give a formulation of the linear independence of mod $p$ 
modular forms.\\
\\
Let the notation and hypotheses be as in \S2.5. 
Recall that $x$ is a closed ordinary point in $V$ with $p^\infty$-level structure $\eta_{p}^\circ$. 
This gives rise to a closed point $\tilde{x}=(x,\eta_{p}^\circ)$ in the Igusa tower $I$ over $x$.
%and a canonical isomomrphism 
%\beq
%\Spf(\widehat{\cO}_{V,x}) \iso \widehat{\mathbb{G}}_{m}
%\eeq
%(\cf Theorem 3.1).\\
%\\
Recall that we have a canonical isomorphism $\widehat{\cO}_{V,x} \iso \widehat{\cO}_{I,\tilde x}$ (\cf (Et)). 
Thus, we have a natural action of the $p$-adic stabiliser $H_{x}(\Zp)$ on $\widehat{\cO}_{I,\tilde x}$ (\cf \S3.2).\\
\\
%Let $f$ be a mod $p$ modular form and $a \in H_{x}(\Z_{p})$. 
Our formulation of the independence is the following.\\
\begin{thm}[Linear independence] 
Let $x$ be a closed ordinary point on the Shimura variety with the local stabiliser $H_x$. 
For $1 \leq i \leq n$, let $a_i \in H_{x}(\Z_p)$ such that 
$a_ia_j^{-1} \notin H_{x}(\Z_{(p)})$ for all $i \neq j$ (\cf \S3.2). 
Let $f_1,...,f_n$ be non-constant mod $p$ modular forms on $V$.
Then, $(a_{i}\circ f_{i})_i$ are linearly independent 
in the Serre-Tate deformation space $\Spf(\widehat{\cO}_{V,x})$.
\end{thm}
\noindent\\
In general, note that $a_{i}\circ f_{i}$ is not a mod $p$ modular form.\\
\\
We actually prove the algebraic independence of $(a_{i}\circ f_{i})_i$.
The theorem is an analogue of Ax-Lindemann-Weierstrass conjecture for the mod $p$ reduction of the Shimura variety.\\
%We are not aware of other formulations of Ax-Lindemann-Weierstrass for mod $p$ Shimura varieties.\\
\\
The theorem is proven in \S4.6.\\
\\
\\
\subsection{Locally stable subvarieties}
In this subsection, we describe the notion and the results regarding locally stable subvarieties of a self-product of the Shimura variety.\\
\\
Let the notation and hypotheses be as in \S 2.5.
Let $n$ be a positive integer. In this subsection, any tensor product is taken $n$ times.\\
\\
Let $Sh$ be as before and $V$ an irreducible component of $Sh$.\\
\begin{defn} A subvariety $Y$ of $V^{n}$ is said to be locally stable if there exists a closed point $y^n=(y,...,y) \in Y$ such that $Y$ is stable under the diagonal action of 
the local algebraic stabilser $H_{y}(\Z_{(p)})$.
\\
\end{defn}
\noindent While considering the problem of linear independence, this type of subvariety arises as follows.\\
\\
We consider an $\F$-algebra homomorphism
\beq \phi_{I}: \cO_{I,\tilde x}\otimes_{\F} ...\otimes_{\F} \cO_{I,\tilde x} \rightarrow \widehat{\cO}_{I,\tilde x} \eeq
given by
\beq f_1\otimes ... \otimes f_n \mapsto \prod_{i=1}^{n}a_i \circ f_i . \eeq
\\
As we are interested in the independence of $(a_i \circ f_i)_i$, we consider $\mathfrak{b}_{I}:= ker(\phi_{I})$.\\
\\
Similarly, we consider an $\F$-algebra homomorphism
\beq \phi=\phi_{V}: \cO_{V,x}\otimes_{\F} ...\otimes_{\F} \cO_{V,x} \rightarrow \widehat{\cO}_{V,x} \eeq
given by
\beq h_1\otimes ... \otimes h_n \mapsto \prod_{i=1}^{n}a_i \circ h_i . \eeq
\\
The following simple observation is crucial in what follows.\\
\\
(EQ) The homomorphism $\phi$ is equivariant with the $H_{x}(\Z_{(p)})$-action.\\
\\
Let $\mathfrak{b}=ker(\phi_{V})$.
\begin{lm} 
Let the notation be as above. 
We have $\mathfrak{b}_{I}=0$ if and only if $\mathfrak{b}=0$.\end{lm}
\begin{proof} In view of $(Et)$, we have an étale
 morphism $\pi^n:  \cO_{V,x}\otimes_{\F} ...\otimes_{\F} \cO_{V,x} \rightarrow \cO_{I,\tilde x}\otimes_{\F} ...\otimes_{\F} \cO_{I,\tilde x}$. 
As $\mathfrak{b}_{I}$ is the unique prime ideal of $\cO_{I,\tilde x}\otimes_{\F} ...\otimes_{\F} \cO_{I,\tilde x}$ over $\mathfrak{b}$, this finishes the proof.\\
\end{proof}
\noindent \\
In view of (EQ), it follows that $\mathfrak{b}$ is a prime ideal of $\cO_{V,x}\otimes_{\F} ...\otimes_{\F} \cO_{V,x}$ 
stable under the diagonal action of $H_{x}(\Z_{(p)})$. Let $X$ be the Zariski closure of $\Spec(\cO_{V,x}\otimes_{\F} ...\otimes_{\F} \cO_{V,x}/\mathfrak{b})$ in $V^n$ . 
Thus, $X$ is a closed irreducible subscheme of $V^n$ containing $x^n$ stable under the diagonal action of $H_{x}(\Z_{(p)})$. In particular, $X$ is a closed irreducible locally stable subvariety of $V^n$.\\
\\
We now axiomatise the above example.\\
\\
(LS1) Let $S_{/\F}$ be $\Spec(\cO_{V,x})$.\\
(LS2) Let $\mathfrak{b}$ be a prime ideal of $\cO_{V,x}\otimes_{\F} ...\otimes_{\F} \cO_{V,x}$ stable under the diagonal action of $H_{x}(\Z_{(p)})$ and 
$\cX_ {/\F}=\Spec(\cO_{V,x}\otimes_{\F} ...\otimes_{\F} \cO_{V,x}/\mathfrak{b})$.\\
(LS3) Let $X$ be the Zariski closure of $\cX$ in $V^n$.\\
\\
Note that $X$ is positive dimensional as $\mathfrak{b}$ is a prime ideal.\\
\\
A natural class of locally stable subvarieties is the following.\\
\begin{defn} 
Let $l$ be a positive integer. For $1 \leq i \leq l$, let $\alpha_i \in H_{x}(\Z_{(p)})$. 
The skewed diagonal $\Delta_{\alpha_1,...,\alpha_l}$ of $V^l$ is defined by $\Delta_{\alpha_1,...,\alpha_l}=\big{\{}(\alpha_1(y),...,\alpha_l(y)): y \in V \big{\}}$.
\end{defn}
\noindent \\
%We note that $\Delta_{\alpha_1,...,\alpha_l}$ is a $H_{x}(\Z_p)$-stable subscheme of $V^l$.\\
%\\
The global structure of a class of locally stable subvarieties is given by the following result.\\
\begin{thm} 
Let $n\geq 2$ and $X$ be locally stable as in (LS1)-(LS3).
Let $S'$ be the product of the first $n-1$ factors of $S^n$ and $S''=S$ be the last factor of $S^n$. Suppose that the projections  
$\pi_{\cX}':\cX \rightarrow S'$ and $\pi_{\cX}'':\cX \rightarrow S''$ are dominant. Then, $X$ equals $V^n$ or $V^{n-2} \times \Delta_{\alpha_1,\alpha_2}$ 
for some $\alpha_1,\alpha_2 \in H_{x}(\Z_{(p)})$ up to a permutation of the first $n-1$ factors.\\
\end{thm}
\noindent The theorem is an instance of Chai-Oort principle that a Hecke stable subvariety of a Shimura variety is a Shimura subvariety (\cf \cite{Ch1}, \cite{Ch2} and \cite{Ch3}). 
The principle is an analogue of Andre-Oort conjecture for mod $p$ Shimura varieties.\\
\\
The theorem is proven in \S4.3-\S4.5.\\
\\
In \cite{Hi3}, an analogue of the theorem is proven in the Hilbert modular case. 
For the proof, we follow the same strategy as in \cite{Ch1}, \cite{Ch2} and \cite[\S3]{Hi3}. 
We first consider the structure of $\Spf(\widehat{X}_{x^n})$ as a formal subscheme of $\Spf(\widehat{\cO}_{V^n,x^n})$. Recall that 
the Serre-Tate deformation space $\Spf(\widehat{\cO}_{V^n,x^n})$ has a natural structure of a formal torus. As $\Spf(\widehat{X}_{x^n})$ is a 
formal subscheme of $\Spf(\widehat{\cO}_{V^n,x^n})$ stable under the diagonal action of the stabiliser $H_{x}(\Z_{(p)})$ and (pretending) $X$ is smooth, 
it follows from Chai's local rigidity that $\Spf(\widehat{X}_{x})$ is in fact a formal subtorus of $\Spf(\widehat{\cO}_{V^n,x^n})$. 
In view of Theorem 3.1, we have a description of $\Spf(\widehat{\cO}_{V^n,x^n})$ as a formal torus.
Based on it, we obtain a description of the formal subtori of $\Spf(\widehat{\cO}_{V^n,x^n})$ stable 
under the diagonal action of the stabiliser $H_{x}(\Z_{(p)})$. 
We thus obtain an explicit description of the possibilities for the structure of $\Spf(\widehat{X}_{x^n})$ as a formal torus. 
Based on de Jong's result on Tate conjecture in \cite{dJ}, we can eliminate all but one possibilities. 
Strictly speaking, the proof is a bit more involved as we cannot directly suppose that $X$ is smooth. 
Instead, we apply the strategy to the normalisation of $X$ and later show that the normalisation equals $X$ itself.\\
\\
\\
\subsection{Local structure} In this subsection, we consider the local structure of a class of locally stable subvarieties of a self-product of the Shimura curve.\\
\\
Let the notation and hypotheses be as in \S 4.2. However, we do not suppose that the projection $\pi_{\cX}''$ is dominant. 
Let $\Pi_{\cY}{}:\cY\rightarrow \cX$ be the normalisation of $\cX$. \\
\\
The following is an analogue of \cite[Prop. 3.11]{Hi3} to our setting.\\
\begin{prop} 
Let the notation and assumptions be as above. Suppose that the projection $\pi_{\cX}':\cX \rightarrow S'$ is dominant. 
Then, the following holds.\\
\\
(1). The normalisation $\cY$ has only finitely many points $y$ over $x^n$. For any such $y$, we have $\widehat{\cY}_y=\widehat{\mathbb{G}}_m \otimes_{\Z_p} L_y$ 
for an $\Z_p$-direct summand $L_y$ of $X_{*}(\widehat{S^n})$. Moreover, 
the isomorphism class of $	L_y$ as an $\Z_p$-module is independent of $y$.\\
(2). The normalisation $\cY$ is smooth over $\F$ and flat over $S'$.\\
(3). Either $\cX=S^n$ or $\cX$ is finite over $S'$. In the later case, the normalisation $\cY$ is finite flat over $S'$.\\
(4). If $\pi_{\cX}' \circ \Pi_{\cY}{}$ induces a surjection of the tangent space at $y$ onto $S'$ at $x^n$ for some $y$ over $x^n$, then 
$\pi_{\cX}' \circ \Pi_{\cY}{}:\cY\rightarrow S'$ is $\acute{e}tale$.\\
\end{prop}
\begin{proof} The proof is similar to the proof of Proposition 3.11 in \cite{Hi3}. Here we only prove (1).\\
\\
Let $\widehat{\cX}=\widehat{\cX}_{x^n}$. By the local stabiliser principle (cf. \cite[Prop. 6.1]{Ch3}), $\widehat{\cX}$ is a formal subscheme of 
$\widehat{S^n}$ stable under the diagonal action of the $p$-adic stabiliser $H_{x}(\Z_p)$. 
In view of Chai's local rigidity (cf. \cite[\S6]{Ch2}), it follows that $\widehat{\cX}=\bigcup_{L\in I} \widehat{\mathbb{G}}_m \otimes_{\Z_p} L$. 
Here $L$ is an $\Z_p$-direct summand of $X_{*}(\widehat{S^n})$ (cf. Theorem 3.1) 
and the union is finite.\\
\\
Note that the semigroup $\End_{SCH}(\cX)$ naturally acts on $\cY$. In particular, the stabiliser $H_{x}(\Z_{(p)})$ acts on $\cY$. 
As $\widehat{\cY}_y$ is integral, the points $y$ are indexed by the irreducible components of $\widehat{\cX}$. Let $y_L$ be the point corresponding to $L$. 
As the morphism $\pi_{\cX}' \circ \Pi_{\cY}{}$ is dominant, there exists at least one $L_{0} \in I$ such that the image of $L_0$ is of finite index in $X_{*}(\widehat{S'})$.\\
\\
To arrive at the desired conclusion, we consider a global argument as follows. Let $\Pi:Y\rightarrow X$ be the normalisation of $X$. 
As before, the local algebraic stabiliser $H_{x}(\Z_{(p)})$ acts on $Y$. 
Recall that the order $\cO_B$ acts on $X$ and hence on $\Theta_X$. This action extends to $\Theta_Y$. 
In view of the Kodaira-Spencer isomorphism for the universal abelian scheme $\cA_{/Y}$ over $Y$ and  the identification $\widehat{Y}_{y_L}=\widehat{\mathbb{G}}_m \otimes_{\Z_p} L$, 
%we have the decomposition
%\beq \epsilon \Theta_Y = \bigoplus_{\mathfrak{p}\in\Sigma_{nr,p}} \epsilon \Theta_{Y,\mathfrak{p}} ,\eeq
%where $\epsilon \Theta_{Y,\mathfrak{p}}$ is the $O_{\mathfrak{p}} \otimes \epsilon\cO_{Y}$-eigen sub-bundle. 
it follows that $\rank_{\Z_p} L = \rank_{{\cO_Y}} \Theta_{Y}$. This finishes the proof of (1).\\
\end{proof}
\noindent\\
%\begin{cor} Suppose $X$ arises as in \S 4.1. If $n=2$ and $\pi_{\cX}'$ is finte, then $\pi_{\cX}' \circ \Pi_{\cY}{}:\cY\rightarrow S'$ is $\acute{e}tale$.\\
%\end{cor}
%\begin{proof} It suffices to verify the hypothesis in Proposition 4.6. This follows as there exists $y$ over $x^2$ such that $\widehat{\cY}_{y} \supset \widehat{\Delta}:=\{(t^{a_{1}},t^{a_{2}})|t\in \widehat{S}\}$ (cf. \cite[pp.85]{Hi3}).
%\end{proof}
%\noindent \\
\\
\subsection{Global structure I} In this subsection, we study global structure of a class of locally stable subvarieties of a $2$-fold self-product of the Shimura curve.\\
\\
Let the notation and hypotheses be as in \S4.2 and $n=2$. In view of Proposition 4.6, 
we conclude that either $\dim \cX= \dim V$ or $X=V^2$. 
From now, we suppose that $\dim \cX= \dim V$. 
In the rest of the subsection, we let $i=1,2$.\\
\\
We need to show that $X$ is a graph of an automorphism of $V$ arising from $H_{x}(\Z_{(p)})$. In other words, we need to show that the projection $\pi_{i}:X\rightarrow V_i$ 
is an isomorphism. 
Here $\pi_i$ is the projection to the $i^{th}$-factor and we view $V_{i}=V$ as the $i^{th}$ factor of $V\times V$. 
Our strategy is the following.\\
\\
{\bf{Step 1.}} The morphism $\Pi_{i}:Y\rightarrow X\rightarrow V$ is étale over an open dense subscheme of $V$. This is proven based on
Serre-Tate deformation theory and Proposition 4.6.\\
\\
{\bf{Step 2.}} Let $\cA_{i}=\Pi_{i}^{*}\cA$ and $\eta_{i}^{(p)}=\Pi_{i}^{*}\eta^{(p)}$. 
There exists an isogeny $\varphi:\cA_{1} \rightarrow \cA_2$. This is proven based on de Jong's theorem in \cite{dJ}. In particular, there exists 
$g\in G(\A_{\Q,f}^{(p)})$ such that $\varphi \circ \eta_{1}^{(p)}=\eta_{2}^{(p)}\circ g$. We consider the fiber at $x$ and conclude that $g\in H_{x}(\Z_{(p)})$. 
Finally, we deduce that $Y=X$ and $X=\Delta_{1,g}$.\\
\\
We now describe each step. In view of Proposition 4.6 and our assumption $\dim \cX= \dim V$, it follows that the morphism $\Pi_{i}:Y\rightarrow V$ is finite at 
any point $y\in Y$ above $x^2$. If $\Pi_i$ is not étale at $y$, then $\Pi_{i,*}(X_{*}(\widehat{Y}_{y})) \subset X_{*}(\widehat{S})$ is a 
$\Z_p$-submodule of finite index. Let $\alpha \in H_{x}(\Z_{(p)})$ such that 
$\alpha^{1-c}X_{*}(\widehat{S})=\Pi_{i,*}(X_{*}(\widehat{Y}_{y}))$. The morphism $\alpha^{-1}\circ \Pi_{i}$ is étale. 
 Thus, from now we suppose that\\
\\
(E) the morphism $\Pi_{i}:Y\rightarrow V$ is étale finite at any point $y\in Y$ over $x^2$.\\
\\
Let $V^{\acute{e}t}$ be the maximal subscheme over which both $\Pi_1$ and $\Pi_2$ are étale. 
Let $\Pi_{i}$ also denote the projection $\cY\rightarrow S$.\\
\begin{lm} Let the notation and assumptions be as above. 
If the projections $\Pi_{1}:\cY\rightarrow S$ and $\Pi_{2}:\cY\rightarrow S$ are both étale, then $V^{\acute{e}t}$ is an open dense subscheme of 
$V^2$ containing $x^2$ and stable under the diagonal action of the stabiliser $H_{x}(\Z_{(p)})$.
\end{lm}
\begin{proof} The later part follows from the definition. 
The former part can be proven in a way similar to the proof of \cite[Prop. 3.14]{Hi3}. 
In fact, the proof is simpler in our case as $Sh$ is proper.\\
\end{proof}
\noindent\\
Let $Y^{ord}=Y\times_{V^2}(V^{ord})^2$. Let $y'\in Y^{ord}$ be a closed point above $(x_1,x_2)\in (V^{ord})^2$.\\
\begin{defn} $Y$ is said to be $\cO_B$-linear at $y'$ if $\Pi_{1}\times\Pi_{2}$ embeds $\widehat{Y}_{y'}$ into $\widehat{V}_{x_{1}}\times\widehat{V}_{x_{2}}$ 
and the equation defining $\widehat{Y}_{y'}$ is given by $t_{1}^{a}=t_{2}^{b}$, where $t_{i}$ is the Serre-Tate co-ordinate of $\widehat{V}_{x_{i}}$ (cf. \S3.1) 
and $a,b \in \End_{\epsilon \cO_{B,p}} (\widehat{S})$.\\
\end{defn}
\noindent Let $Y^{lin}\subset Y^{ord}$ be the subset of all $\cO_B$-linear points. Here is the Step 1 outlined above.\\
\begin{prop} Let the notation and assumptions be as above.  
The subset $Y^{lin}$ contains the closed points of an open dense subscheme $U$ of $Y^{ord}$.
\end{prop}
\begin{proof} We first describe the main steps of the proof.\\
\\
{\it{$\cO_B$-linearity at $y$}} : In view of Proposition 4.6, it follows that $Y$ is $\cO_B$-linear at $y$. In fact, we can suppose that $b=1$ and $a=a_{x}$ is a unit in 
$\End_{\epsilon \cO_B} (\widehat{S})$. We regard $a_x$ as an isomorphism 
$a_{x}: \epsilon\cA_{1,x}[p^\infty]\rightarrow \epsilon\cA_{2,x}[p^\infty]$ 
of Barsotti-Tate groups. {\it{Existence of $U$ via extension of a homomorphism}} : Based on Serre-Tate deformation theory, 
we reduce the existence of $U$ to the existence of an étale covering $\widetilde{U}$ of an open dense subscheme of $Y^{ord}$ containing $y$ such that $a_x$ 
extends to a homomorphism 
$\tilde{a}:\epsilon\cA_{1}[p^\infty]_{/\widetilde{U}}\rightarrow \epsilon\cA_{2}[p^\infty]_{/\widetilde{U}}$ of Barsotti-Tate groups over $\widetilde{U}$. 
{\it{Extension of a homomorphism}} : Based on Serre-Tate deformation theory,
 we first show that $a_x$ extends to a homomorphism 
 $\widehat{a}:\epsilon\cA_{1}[p^\infty]_{/\widehat{Y}_{y}}\rightarrow \epsilon\cA_{2}[p^\infty]_{/\widehat{Y}_{y}}$ 
 of Barsotti-Tate groups over $\widehat{Y}_{y}$. The existence of $\tilde{U}$ with the desired property then follows by an $fpqc$-descent.\\
\\
We now describe each step.\\
\\
{\it{$\cO_B$-linearity at $y$}} : By part (1) of Proposition 4.6, the formal torus $\widehat{\cY}_y$ is canonically isomorphic to a formal subtorus 
$\widehat{\mathbb{G}}_{m} \otimes L$ of $\widehat{S}^{2}= \widehat{\mathbb{G}_{m}^{2}}$ for  
a $\Z_p$-free direct summand $L$ of $\Z_{p}^{2}$. 
As $\dim{\cX}=\dim{V}$, 
we conclude that the equation of $\widehat{\cY}_y$ in $\widehat{S}^{2}$ is given by $t_{1}^{a}=t_{2}^{b}$, for some $a,b \in \Z_{p}$ 
such that $a \Z_{p}+b\Z_{p}=\Z_{p}$. Moreover, 
\beq L=\bigg{\{}(c,d)\in \Z_{p}^{2}:ac=bd\bigg{\}}.\eeq
In particular, $Y$ is $\cO_B$-linear at $y$. In view of (E), we replace $(a,b)$ by $(a/b,1)$ and suppose that $(a,b)=(a_x,1)$ with 
$a_{x} \in \Z_{p}^{\times}$. We now normalise the action of 
$a \in \Z_{p}$ on $\widehat{S}$ as follows. We modify $a$ by an element in 
$\Z_{p}$ if necessary, such that $a$ acts as identity on $\epsilon A_{x}[p^\infty]^{\circ}$ 
without affecting the original action of $a$ on $\widehat{S}$. In other words, $a$ acts on $\widehat{S}$ via the action on 
$\epsilon A_{x}[p^\infty]^{\acute{e}t}$. We identify $\cA_{i,y}$ with $A_x$. 
In this way, we regard $a_x$ as an isomorphism $a_{x}:\epsilon\cA_{1,x}[p^\infty]\rightarrow \epsilon\cA_{2,x}[p^\infty]$ of Barsotti-Tate groups.\\
\\
{\it{Existence of $U$ via extension of a homomorphism}} : Suppose that there exists an open dense subscheme 
$U\subset Y^{ord}$ containing $y$ having an irreducible étale cover $\tilde{U}$ such that $a_x$ extends to an isomorphism 
$\tilde{a}:\epsilon\cA_{1}[p^\infty]_{/\widetilde{U}}\rightarrow \epsilon\cA_{2}[p^\infty]_{/\widetilde{U}}$ 
of Barsotti-Tate groups over $\widetilde{U}$. We show that $Y$ is $\cO_B$-linear at the closed points of $U$. In view of (E) and 
Lemma 4.7, shrinking the neighbourhood $U$ of $y$ if necessary, we suppose that the projection $\Pi_{i}:U\rightarrow V$ is étale. 
Let $u\in U$ and $(\Pi_{1}(u),\Pi_{2}(u))=(u_1,u_2)$. In view of Serre-Tate deformation theory (cf. \S3.1) and the isomorphism 
$\pi_{i}:\widehat{Y}_{u}\iso \widehat{V}_{u_i}$, it follows that $\widehat{Y}_u$ is isomorphic to the deformation space of Barsotti-Tate $\epsilon \cO_{B,p}$-module 
$\epsilon A_{u_{i}}[p^\infty]$ via $\Pi_i$. 
Note that the universal deformation of the Barsotti-Tate $\epsilon \cO_{B,p}$-module 
$\epsilon A_{u_{i}}[p^\infty]$ is the Barsotti-Tate $\epsilon \cO_{B,p}$-module 
$\epsilon \cA_{i}[p^\infty]_{/\widehat{Y}_{u}}\iso \epsilon \cA[p^\infty]_{/\widehat{V}_{u_i}}$. 
We choose a $p^\infty$-level structure 
$\eta_{i,p}$ on $\epsilon A_{u_{i}}[p^\infty]$ for $i=1,2$ (cf. (PL)). 
Accordingly, we have the Serre-Tate co-ordinates $t_{i}$ (cf. (CC) and (STC)). 
The composition $\tilde{a}\circ \eta_{1,p}$ gives rise to a possibly new $p^\infty$-level structure on 
$\epsilon A_{u_{2}}[p^\infty]$. 
In view of the fact that $\Aut_{\epsilon \cO_{B,p}}{\widehat{Y}_u}=\Z_{p}^{\times}$ and Definition 3.2, it follows that the new $p^\infty$-level structures 
give rise to the Serre-Tate co-ordinate $t_{2}^{a_u}$ for some $a_{u} \in \Z_{p}^{\times}$. Thus, 
we have $t_{1}=t_{2}^{a_u}$. In other words, $\widehat{Y}_{u}$ is contained in the formal subscheme of $\widehat{V}_{u_1}\times \widehat{V}_{u_2}$ defined by 
$t_{1}=t_{2}^{a_u}$. As $\widehat{Y}_{u}$ is a smooth formal subscheme with an isomorphism $\pi_{i}:\widehat{Y}_{u}\iso \widehat{V}_{u_i}$, 
we conclude that $\widehat{Y}_{u}$ is itself defined by the equation $t_{1}=t_{2}^{a_u}$.\\
\\
{\it{Extension of a homomorphism}} : We first extend $a_x$ to a homomorphism 
$\widehat{a}:\epsilon\cA_{1}[p^\infty]_{/\widehat{Y}_{y}}\rightarrow \epsilon\cA_{2}[p^\infty]_{/\widehat{Y}_{y}}$ of Barsotti-Tate groups over $\widehat{Y}_{y}$. 
Let $n$ be a positive integer. 
Let $\alpha_{n} \in H_{x}(\Z_{(p)})$ be a prime to $p$ element such that $\alpha_{n}^{1-c} \equiv a$ (mod $p^{n}$). 
The homomorphism $\alpha_n$ also induces a homomorphism $\alpha_n:\cA_{/V} \rightarrow \cA_{/V}$ and hence, a homomorphism
\begin{diagram} \alpha_n:\epsilon \cA_{1}[p^n]_{/\widehat{Y}} &\rTo^{\iso}_{\Pi_1} & \epsilon \cA[p^n]_{/\widehat{V}_x} &\rTo_{\alpha_n} &\epsilon \cA[p^n]_{/\widehat{V}_x} &\rTo^{\iso}_{\Pi_{2}^{-1}} 
&\epsilon \cA_{2}[p^n]_{/\widehat{Y}}.
\end{diagram}
Thus, we have a homomorphism $\alpha_{n}:\epsilon \cA_{1}[p^n]_{/\widehat{Y}_{n}}\rightarrow \epsilon \cA_{2}[p^n]_{/\widehat{Y}_{n}}$ 
for an infinitesimal neighbourhood $\widehat{Y}_n$ of $y$. 
In view of Lemma 3.5, it follows that the homomorphism $\alpha_n$ equals $a|_{\epsilon \cA_{1}[p^n]}$ over $\widehat{Y}_n$. 
Let $\widehat{a}=\varprojlim \alpha_{n}|_{\epsilon \cA_{1}[p^n]_{/\widehat{Y}_n}}$. 
By definition, $\widehat{a}:\epsilon\cA_{1}[p^\infty]_{/\widehat{Y}_{y}}\rightarrow \epsilon\cA_{2}[p^\infty]_{/\widehat{Y}_{y}}$ 
is a homomorphism of Barsotti-Tate groups over $\widehat{Y}_{y}$. 
The existence of $\tilde{U}$ with the desired can be proven from the existence of $\widetilde{a}$ by an $fpqc$-descent by an argument similar to \cite[Prop. 8.4]{Ch2}
and \cite[pp. 92-93]{Hi3}.\\
\end{proof}
\noindent\\
Here is the Step 2 outlined in the introduction of this subsection.\\
\begin{thm} Let $X$ be as before. Then, $X$ is smooth and there exist $\alpha,\beta\in H_{x}(\Z_{(p)})$ such that 
$X=\Delta_{\alpha,\beta}$. Moreover, if (E) holds, then we can take $(\alpha,\beta)=(1,\beta)$ with $\beta$ being a $p$-unit.
\end{thm}
\begin{proof} Recall that we have dominant projections $\Pi_{i}:Y\rightarrow V_{i}$. It follows that 
$\End(\cA_i)\otimes \Q=B$. In view of the discussion before (E), we suppose that (E) holds.\\
\\
Let $\cA_{/Y}=\cA_{1}\times_{Y}\cA_2$ and $\End^{\Q}(\cA)=\End(\cA_{/Y})\otimes \Q$. 
We now have two possibilities for $\End^{\Q}(\cA)$, namely $B^2$ or $M_2(B)$. 
Based on de Jong's theorem in \cite{dJ}, we first show that $\End^{\Q}(\cA)=M_2(B)$. 
Along with Serre-Tate deformation theory, we later finish the proof.\\
\\
We first suppose that $\End^{\Q}(\cA)=B^2$. We proceed in a way similar to the proof of Theorem 4.6. 
Let $U$ be as in Proposition 4.9. In view of (E) and (PL), it follows that
\beq \End_{\epsilon \cO_{B,p}}{\epsilon \cA[p^\infty]_{/U}}=M_{2}(\cO_{M_{0}}).\eeq
By Theorem 2.6 of \cite{dJ}, we have 
\beq \End_{\cO_{B,p}}(\cA[p^\infty]_{/U}) = \End_{\cO_B}(\cA_{{/U}}) \otimes \Z_{p}.
\eeq
This contradicts our assumption that $\End^{\Q}(\cA)=B^2$. Thus, $\End^{\Q}(\cA)=M_2(B)$.\\
\\
Let $e_{i} \in \End^{\Q}(\cA)$ be two commuting idempotents such that $e_{i}(\cA)=\cA_{i/Y}$. As $\End^{\Q}( \cA)=M_2(B)$, 
there exists $\tilde{\beta} \in GL_{2}(\cO_{B,(p)})$ such that $\tilde{\beta}\circ e_1 = e_2$. 
Thus, we have as isogeny $\tilde{\beta}:\cA_{1}\rightarrow \cA_{2/Y}$. 
In particular, there exists $g\in G(\A_{\Q,f}^{(p)})$ such that $\varphi \circ \eta_{1}^{(p)}=\eta_{2}^{(p)}\circ g$. 
Specialising to the fiber of $\cA_{i}$ at $y$, 
we conclude that $g$ is induced by an element $\beta \in H_{x}(\Z_{(p)})$. Thus, the equation defining 
$\widehat{Y}_{y} \subset \widehat{S} \times \widehat{S}$ is given by $t_{2}=t_{1}^{\beta^{1-c_{x}}}$. 
In view of (E) and the argument in $\cO_{B}$-linearity step of the proof of Proposition 4.9, we conclude that $\beta^{1-c_{x}}\in \Z_{(p)}^{\times}$. 
Here $c_{x}$ denotes the non-trivial element in the Galois group of the imaginary quadratic extension over the rationals associated to $x$. 
Thus, we suppose that $\beta \in \Z_{p}^{\times}$. 
Recall that in the 
proof of (1) of Proposition 4.6 we have $\widehat{X}=\bigcup_{L\in I} \widehat{\mathbb{G}}_m \otimes_{\Z_p} L$, where $L$ is an 
$\Z_{p}$-direct summand of $X_{*}(\widehat{S^2})$ and the union is finite. 
Moreover, the points of $Y$ above $x^2$ are indexed by $L\in I$. More precisely, if $y$ corresponds to $L$, then $\widehat{Y}_{y}=\widehat{\mathbb{G}}_m \otimes_{\Z_p} L$. 
Thus, $\widehat{\Delta_{1,\beta}}_{,x^2}=\widehat{Y}_{y}\subset \widehat{X}$. 
By an $fpqc$-descent, it follows that $\Delta_{1,\beta}\subset X$. As $X$ is irreducible, we conclude that $\Delta_{1,\beta}= X$.\\
\end{proof}
\noindent\\
Thus, we have proven Theorem 4.5 for $n=2$.\\
\\
\\
\subsection{Global structure II} In this subsection, 
we study a class of global structure of locally stable subvarieties of $n$-fold self-product of the Shimura curve for $n\geq 2$.\\
\\
Let the notation and hypotheses be as in \S4.2.\\
\\
Moreover, we suppose that $\cX\neq S^n$.\\
\begin{thm} 
Let $X$ be as before. Then, $X$ is smooth and there exist $\alpha,\beta\in H_{x}(\Z_{(p)})$ such that $X=V^{n-2} \times \Delta_{\alpha,\beta}$ up to a permutation of the first $n-1$ factors.
\end{thm}
\begin{proof} The proof of Proposition 4.9 also works for $n\geq 2$.\\
\\
Let $\cA_{/Y}=\cA^{n}\times_{V^n}Y$ and $\End^{\Q}(\cA)=\End( \cA_{/Y})\otimes \Q$. 
We now have two possibilities for $\End^{\Q}(\cA)$, namely $B^n$ or $B^{n-2}\times M_2(B)$. In a same way as in the proof of Theorem 4.10, 
it can be shown that $\End^{\Q}(\cA)=B^{n-2}\times M_2(B)$. Thus, there exists $i<n$ such that $i^{th}$ factor $\cA_{i}$ of $\cA_{/Y}$ 
is isogenous to the last factor $\cA_n$ of $\cA_{/Y}$. \\
\\
Based on Serre-Tate deformation theory, we finish the proof in the same way as the proof of Theorem 4.10.\\
\end{proof}
\noindent\\
This finishes the proof of Theorem 4.5.\\
\\
Theorem 4.11 can also be proven by induction on $n$ (cf. similar to the proof of \cite[Cor. 3.19]{Hi3}).\\
\\
\\
\subsection{Linear independence} In this subsection, 
we prove the linear independence of mod $p$ modular forms as formulated in \S4.1 based on the global structure of locally stable subvarieties.\\
\\
As before, we have the following independence.\\
\begin{thm} 
Let $x$ be a closed ordinary point on the Shimura variety with the local stabiliser $H_x$. 
For $1 \leq i \leq n$, let $a_i \in H_{x}(\Z_p)$ such that 
$a_ia_j^{-1} \notin H_{x}(\Z_{(p)})$ for all $i \neq j$ (\cf \S3.2). 
Let $f_1,...,f_n$ be non-constant mod $p$ modular forms on $V$.
Then, $(a_{i}\circ f_{i})_i$ are algebraically independent 
in the Serre-Tate deformation space $\Spf(\widehat{\cO}_{V,x})$.
\end{thm}
\begin{proof} From the definition of $\mathfrak{b}$ (cf. \S4.1), it suffices to show that $\mathfrak{b}=0$. This is equivalent to show that $X=V^n$.\\
\\
When $n=1$, this follows from the definition. Let $n\geq2$. 
We first note that there exists $y$ over $x^n$ such 
that 
$$\widehat{\cY}_{y} \supset \widehat{\Delta}:=\big{\{}(t^{a_{1}},t^{a_{2}},...,t^{a_{n}})|t\in \widehat{S}\big{\}}$$ 
(cf. similar to  \cite[pp.85]{Hi3}). 
When $n=2$, this verifies the hypothesis in Theorem 4.5. When $n\geq2$, the hypothesis in the theorem can be verified by induction on $n$ up to a permutation of the first $n-1$ factors.\\
\\
In view of the theorem, the locally stable subscheme $X$  equals $V^n$ or $V^{n-2} \times \Delta_{\alpha,\beta}$ for some $\alpha,\beta\in H_{x}(\Z_{(p)})$ up to a permutation of the 
first $n-1$ factors. We first suppose that $X=V^{n-2} \times \Delta_{\alpha,\beta}$. Let $\cX'$ be the projection of $\cX$ to the last two factors 
of $S^n$. As $X=V^{n-2} \times \Delta_{\alpha,\beta}$, the equation of $\widehat{\cX'}_{x^2} \subset \widehat{S^2}$ is given by 
$t^{\beta^{1-c_{x}}}=(t')^{\alpha^{1-c_{x}}}$ for $t$ (resp. $t'$) being the Serre-Tate co-ordinate of the second last (resp. last) factor of $\widehat{S^n}$. 
On the other hand, it follows from the definition of $X$ that the equation is also given by $t^{a_{n}}=t'^{a_{n-1}}$. Thus, $a_{n}a_{n-1}^{-1}=(\beta\alpha^{-1})^{1-c_{x}} \in H_{x}(\Z_{(p)})$. 
This contradicts our hypothesis on $a_i$'s and finishes the proof.\\
\end{proof}
\noindent
\begin{remark} 
The independence also plays a key role in \cite{Bu3}. 
In \cite{Bu3}, we consider an $(l,p)$-analogue of the non-triviality of generalised Heegner cycles modulo $p$.
\end{remark}
%In particular, if $(f_i)_i$'s are algebraically independent in ${\cO}_{I,\tilde x}$ over $\F$, then $(a_i(f_i))_i$'s are in fact algebraically independent in $\widehat{\cO}_{I,\tilde x}$ over $\F$.\\
\noindent\\
\\
\section{$\mu$-invariant of Anticyclotomic Rankin-Selberg $p$-adic L-functions}
\noindent In this section, we prove the vanishing of the $\mu$-invariant of a class of anticyclotomic Rankin-Selberg $p$-adic L-functions. 
In \S5.1, we consider the $p$-depletion of a normalised Jacquet-Langlands transfer of an elliptic Hecke eigenform. 
In \S5.2, we describe generalities regarding the anticyclotomic Rankin-Selberg $p$-adic $L$-functions. 
In \S5.3, we prove the vanishing of the $\mu$-invariant.\\
\\
\subsection{$p$-depletion} In this subsection, we consider the $p$-depletion of a normalised Jacquet-Langlands transfer of an elliptic Hecke eigenform.\\
\\
Let the notation and hypotheses be as in the introduction. 
Let $f \in S_{k}(\Gamma_{0}(N),\epsilon)$ be an elliptic newform such that:\\
\\
(irr) the residual Galois representation $\rho_{f}$ modulo $p$ is absolutely irreducible.\\
\\
The Jacquet-Langlands correspondence implies the existence of a classical modular form $f_{B}$ on $Sh_{B}$ such that the following holds:\\
\\
(JL1) $f_{B}$ is a classical modular form on the Shimura curve $Sh_{B}$ of weight $k$ and neben-character $\epsilon$.\\
(JL2) For positive integer $n$ such that $(n,N^{-})=1$, $f_{B}$ is a Hecke eigenform for the Hecke operator $T_{n}$ with the same eigenvalue as $f$.\\
\\
We normalise $f_{B}$ by requiring that 
\beq
\mu(f_{B})=0.
\eeq
In what follows, we consider $f_{B}$ as being defined over $\cO_{E_{f},\mathfrak{P}}$ and 
regard it as a $p$-adic modular form on the Shimura variety $Sh$. Here $E_{f}$ is the Hecke field and $\mathfrak{P}$ a prime above $p$ induced by $\iota_p$.\\
\\
We have the following useful lemma.\\
\begin{lm} 
Let the notation and assumptions be as above. Then, 
the Hecke eigenform $f_{B}$ is non-constant modulo $p$.\\
\end{lm}
\begin{proof} 
If the Hecke eigenform $f_{B}$ is constant modulo $p$, then the Hecke eigensystem is Eisenstein. 
This contradicts (irr).\\
\end{proof}
\noindent\\
\\
A class of anticyclotomic Rankin-Selberg $p$-adic L-functions is constructed via toric periods of the $p$-depletion of $f_{B}$.\\
\\
Let 
\beq
f_{B}^{(p)}=f_{B}|_{VU-UV}
\eeq
be the $p$-depletion of $f_{B}$. Here $V$ and $U$ are Hecke operators in \cite[\S3.6]{Br}.\\
\\
We recall the following lemma.\\
\begin{lm} 
Let the notation be as above. 
Let $f_{B}(t)=\sum_{\xi \in \Z_{\geq 0}} a(\xi,f_{B}) t^{\xi}$ be the $t$-expansion of $f_{B}$ around $x$.
We have 
$$f_{B}^{(p)}(t)= \sum_{\xi \in \Z_{\geq 0}, p\ndivide \xi} a(\xi,f_{B}) t^{\xi}.$$
\end{lm}
\begin{proof} 
This readily follows from \cite[Prop. 4.17]{Br}.\\
\end{proof}
\noindent\\
We have the following proposition regarding the $p$-integrality of the $p$-depletion.\\
\begin{prop} 
Let the notation and assumptions be as above. 
We have 
$$ \mu(f_{B}^{(p)})=0.$$ 
Moreover, the $p$-depletion is non-constant modulo $p$.\\
\end{prop}
\begin{proof} 
The proof is based on the $p$-integrality criterion in \cite[Prop. 2.9]{Pr}.\\
\\
In view of the criterion and (5.1), we have 
\beq
\min_{M,\chi}\mu\bigg{(}\frac{L_{\chi}(f_{B})}{\Omega_{M}^{2k}}\bigg{)} = 0.
\eeq 
Here $M$ is an imaginary quadratic extension of $\Q$ such that it has an embedding into the indefinite quaternion algebra $B$, $p$ splits in $M$ and 
$p$ does not divide the class number of $M$. Moreover, $\chi$ is an unramified Hecke character over $M$ with infiinity type $(k,0)$, 
$L_{\chi}(f_{B})$ the toric period of the pair $(f,\chi)$ and $\Omega_{M}$ the CM period (\cf \cite[\S2.3]{Pr}).\\
\\
As $f_{B}$ is a Hecke eigenform, we have
\beq
L_{\chi}(f_{B}^{(p)})=(1-\chi^{-1}(\overline{\mathfrak{p}})a_{p}+\chi^{-2}(\overline{\mathfrak{p}})\epsilon(p)p^{k-1})L_{\chi}(f_{B})
\eeq 
(\cf \cite[Prop. 8.9]{Br}). 
Here $\mathfrak{p}$ is a prime above $p$ induced by the $p$-adic embedding $\iota_p$ as before and $a_{p}$ denotes the $T_{p}$-eigenvalue of $f$.\\
\\
It thus follows that
\beq
\min_{M,\chi}\mu\bigg{(}\frac{L_{\chi}(f_{B}^{(p)})}{\Omega_{M}^{2k}}\bigg{)} = 0.
\eeq
In view of the criterion, this finishes the proof.\\
\\
``Moreover'' part now immediately follows from the previous lemma.\\
\end{proof}
\noindent\\
For later purposes, we introduce the following modular forms related to the $p$-depletion.\\
\\
Let $\cU_{p}$ be the torsion subgroup of $\Z_{p}^\times$. For $u \in \cU_{p}$, let $\phi_{u}:\Z/p\Z \rightarrow \cW$ be the indicator function corresponding to $u$.\\
\\
Let $f_{B,u}$ be the modular form given by
\beq
f_{B,u}=f_{B}|\phi_{u}. 
\eeq
\begin{lm} 
Let the notation and assumptions be as above. 
We have 
\beq
f_{B}^{(p)}=\sum_{u\in \cU_{p}} f_{B,u}.
\eeq
\end{lm}
\begin{proof}
This follows from the $t$-expansion principle, Proposition 3.6 and Lemma 5.2.\\
\end{proof}
\noindent\\
\\
\subsection{Anticyclotomic Rankin-Selberg $p$-adic L-functions} In this section, we describe generalities regarding a class of anticyclotomic Rankin-Selberg $p$-adic 
$L$-functions. This is a slight reformulation of the results in \cite[\S8]{Br}. 
We adapt the formulation in \cite[\S1.3.7]{Hi5}, \cite[\S8]{Br1} and \cite{Hs3}.\\
\\
Let the notation and hypotheses be as in the introduction. 
Let $\Gamma_{N^{+}} = \varprojlim G_{p^n}$, where $G_{p^{n}}=\Gal(H_{N^{+}p^{n}}/K)$. 
Recall that $\Gamma$ denotes the $\mathfrak{\Z}_{p}$-quotient of $\Gamma_{N^{+}}$. 
We fix a section of the projection $\pi: \Gamma_{N^{+}} \twoheadrightarrow \Gamma$ stable under the action of complex conjugation $c$. 
Let $\cC(\Gamma,\overline{\Z}_{p})$ be the space of continuous functions on $\Gamma$ with values in $\overline{\Z}_{p}$.\\
\\
Let $\Sg_{cc}$ denote the set of Hecke character over $K$ central critical for $f$ i.e. the set of Hecke characters $\lam$ such that $\lam$ of infinity type $(j_{1},j_{2})$ with 
$j_{1}+j_{2}=k$ and $\epsilon_{\lam}=\epsilon_{f}{\bf{N}}^{k}$. 
Let $\Sg_{cc}^{(1)}$ be the subset of Hecke character of $K$ with infinity type $(l_{1},l_{2})$ such that 
$1 \leq l_{1},l_{2} \leq k-1$. Let $\Sg_{cc}^{(2)}$ be the subset of Hecke character over $K$ with infinity type $(l_{1},l_{2})$ such that 
$l_{1} \geq k$ and $l_{2} \leq 0$. Recall that $\mathfrak{X}$ denotes the set of anticyclotomic Hecke characters over $K$ factoring though $\Gamma$.\\
\\
For $\chi \in \Sg_{cc}^{(1)}$ (resp. $\Sg_{cc}^{(2)}$), the global root number of the Rankin-Selberg convolution $L(f ,\chi^{-1},s)$ equals $-1$ (resp. $1$). 
From now, we fix an unramified Hecke character $\chi \in \Sg_{cc}^{(2)}$ with infinity type $(k, 0)$. 
For $\eta \in \Gamma$, note that $\chi\eta \in \Sg_{cc}^{(2)}$.\\
\\
Let $$Cl_-:=K^\x\A_{\Q,f}^{\times}\bksl \A_{K,f}/U_K$$ 
and $Cl_-^{alg}$ the subgroup of $Cl_-$ generated by ramified primes. 
Here $U_{K}=(K\otimes_{\Q}\R)^{\times} \times (\cO \otimes_{\Z}\widehat{\Z})^\times$. 
Let $\mathcal{R}$ be the subgroup of $\A_{K}^\times$ generated by $K_{v}^\times$ for ramified $v$. 
Let $Cl_{-}^{alg} \subset Cl_{-}$ be the subgroup generated by $\mathcal{R}$. Let $\mathcal{D}_{1}$ be a set of representatives for $Cl_{-}/Cl_{-}^{alg}$ in $(\A_{K,f}^{(pN)})^\times$.
%Recall that $\cU_p$ denotes the torsion subgroup of $\Zp^\x$ and 
Let $\cU^{alg}=U_K\cap (K^\x)^{1-c}$.\\
\\
For $a\in \cD_{1}$, let $\cF_{B}^{(p)}(x(a))$ be the $p$-adic measure on $\Gamma$ such that 
$$ \int_{\Gamma} \binom{x}{n}d\cF_{B}^{(p)}(x(a))= \binom{d}{n} f_{B}^{(p)}(x(a)).                              $$
Here $n$ is a non-negative integer and $d$ denotes the Katz $p$-adic differential operator.\\
\\
We have the following useful result.\\
\begin{lm} 
The power series expansion of the measure $\cF_{B}^{(p)}(x(a))$ regarded as a $p$-adic measure on $\Z_p$ with support in $1+p\Z_p$ equals 
the $t$-expansion of the $p$-depletion $f_{B}^{(p)}$ around the CM point $x(a)$.\\
\end{lm}
\begin{proof} 
This follows from a similar argument as in the proof of \cite[Prop. 8.1]{Br1}.\\
\end{proof} 
\noindent\\
Let $L_{p}(f,\chi)$ be the $p$-adic measure on $\Gamma_{\mathfrak{N}^{+}}$ such that for $\varphi \in \mathcal{C}(\Gamma,\overline{\Z}_{p})$, we have 
\beq
\int_{\Gamma_{N^{+}}} \varphi dL_{p}(f,\chi) = \sum_{a \in \mathcal{D}_{1}} \chi(a)\int_{\Gamma} \varphi |[a] df_{B}^{(p)}(x(a)).
\eeq
The operator $|[a] \in \End(\mathcal{C}(\Gamma,\overline{\Z}_{p}))$ is given by 
$\varphi \mapsto \varphi|[a](\sg)=\varphi(\sg\rec_{K}(a)|_{\Gamma})$.\\
\\
To state the interpolation property of the measure, we need further notation.\\
\\
Let 
\beq
\alpha(f,f_{B})=\frac{\langle f, f \rangle}{\langle f_{B}, f_{B} \rangle}. 
\eeq
Here $\langle \cdot , \cdot \rangle$ denotes the normalised Petersson inner product in \cite[\S1]{Pr}.\\
\\
For an unramified Hecke character $\lam \in \Sg_{cc}^{(2)}$ with infinity type $(k+j, -j)$ for $j \geq 0$, let
\beq
C(f,\lam)=\frac{1}{4}\pi^{k+2j-1}\cdot \Gamma(j+1)\Gamma(k+j) \cdot w_{K} \sqrt{|d_{K}|}2^{|S_{f}|}\cdot \prod_{l|N^{-}}\frac{l-1}{l+1}.
\eeq
Here $S_{f}$ denotes the set of primes which ramify in $K$ that divide $N^{+}$ but do not divide the conductor of $\epsilon$.\\
\\
Let $\mathfrak{b}\subset \cO_{K}$ be an ideal, $b_{N}\in \cO_{K}$ and $W_{f} \in \C^\times$ as in \cite[Prop. 8.3]{Br} and following Corollary 8.4 of \cite{Br}, respectively.\\
\\
Let
\beq
W(f,\lam)=W_{f} \cdot \epsilon_{f}({\bf{N}})(\mathfrak{b})^{-1}\lam {\bf{N}}^{-j}(\mathfrak{b}) \cdot (-N)^{k/2+j}b_{N}^{-k-2j}. 
\eeq 
Let $(\Omega, \Omega_{p}) \in \C^\times \times \C_{p}^\times$ be the complex and $p$-adic CM periods in the beginning of \cite[\S8.4]{Br}.\\
\\
We have the following result regarding the $p$-adic variation of central critical Rankin-Selberg L-values over the $\Z_{p}$-anticyclotomic extension of $K$.\\
\begin{thm} (Brooks) 
Let the notation be as above. 
Let $f \in S_{k}(\Gamma_{0}(N),\epsilon)$ be an elliptic newform and $\chi \in \Sg_{cc}^{(2)}$ an unramified Hecke character over $K$ with infinity type $(k,0)$. 
For an unramified Hecke character $\nu \in \mathfrak{X}$ with infinity type $(m,-m)$, we have 
$$    \frac{\widehat{\nu}(L_{p}(f,\chi))}{\Omega_{p}^{2(k+2(a+m))}} = 
(1-(\chi\nu)^{-1}(\overline{\mathfrak{p}})a_{p}+ (\chi\nu)^{-2}(\overline{\mathfrak{p}})\epsilon(p)p^{k-1}) \cdot t_{K} \cdot\frac{C(f,\chi\nu)}{\alpha(f,f_{B}) W(f,\chi\nu)} \cdot \frac{L(f,\chi\nu,0)}{\Omega^{2(k+2m)}}.            $$
Here 
$$  t_{K}=\frac{|\mathcal{U}^{alg}|}{[\cO^{\times}:\Z^{\times}]|Cl_{-}^{alg}|}                     $$
\end{thm}
\begin{proof} 
This is essentially proven in \cite[Prop. 8.9]{Br} based on the Waldspurger formula on the Shimura curve. 
The extra factor $t_{K}$ arises from the definition (5.8).\\
\end{proof}
\noindent
\begin{remark} 
(1). Note that $t_{K}$ equals a power of $2$.\\
\\
(2). We have an analogous construction of the $p$-adic L-function for unramified Hecke characters $\chi$ over $K$ with infinity type $(k+a,-a)$ 
for $a \geq 0$.\\
 \end{remark}
\noindent\\
\\
\subsection{$\mu$-invariant} 
In this subsection, we prove the vanishing of the $\mu$-invariant of a class of antiyclcotomic Rankin-Selberg $p$-adic L-function.\\
\\
Let the notation and hypotheses be as in \S5.2. 
Let $\Gamma'$ be the open subgroup of $\Gamma$ generated by the image of $1+p\Z_p$ via $\rec_K$. 
%The reciprocity law $\rec_K$ at $p$ induces an injective map $\rec_{\Sg_p}\colon 1+p\Z_{p}\hookto \Z_{p}^\x\stackrel{\rec_K}\longto Z(\frakC)^-$ 
%with finite cokernel and it is easy to see that $\rec_{\Sg_p}$ induces an isomorphism $\rec_{\Sg_p}:1+p\Z_p\isoto\Gamma'$. 
%We thus identify $\Gamma'$ with the subgroup $\rec_{\Sg_p}(1+p\Z_p)$ of $Z(\frakC)^-$. 
Let $\pi_{-}: (\A_{K,f}^{(pN)})^\times \rightarrow \Gamma$ be the map arising from the reciprocity law. 
Let $Z':=\pi_{-}^{-1}(\Gamma')$ be the subgroup of $(\A_{K,f}^{(pN)})^\times$
and let $Cl'_-\supset Cl^{alg}_-$ be the image of $Z'$ in $Cl_-$ and let $\cD'_1$ (resp. $\cD_1''$) be a set of representatives of $Cl'_-/Cl^{alg}_-$ (resp. $Cl_-/Cl'_-$) in 
$(\A_{K,f}^{(pN)})^\x$. Let $\cD_1:=\cD_1''\cD'_1$ be a set of representatives of $Cl_-/Cl^{alg}_-$. 
Let $\cD_0$ be a set of representatives of  $\cU_p/\cU^{alg}$ in $\cU_p$.\\
\\
We have the following theorem regarding the $\mu$-invariant of the anticyclotomic Rankin-Selberg $p$-adic L-functions.\\
\begin{thm} 
Let $f \in S_{k}(\Gamma_{0}(N),\epsilon)$ be an elliptic newform and $\chi \in \Sg_{cc}^{(2)}$ an unramified Hecke character over $K$ with infinity type $(k,0)$. 
Suppose that the hypotheses (ord), (H1), (H2), (H3) and (irr) hold. Then, we have 
$$  \mu(L_{p}(f,\chi))=0.       $$
\end{thm}
\begin{proof}
%Let $\bfx$ be the CM point corresponding to the trivial ideal class. 
Let $t$ denote the Serre-Tate co-ordinate of the deformation space of the CM point $x(1)$ corresponding to the trivial ideal class. 
For $a\in \cD_1'$, let $\langle a \rangle_{\Sigma}$ be the unique element in $1+p\Z_p$ such that 
$\rec_{\Sg_p}(\langle a \rangle_{\Sigma})=\pi_-(\rec_\cK(a))\in\Gamma'$.\\
\\
For $(a,b) \in \cD_{1} \times \cD_{1}^{''}$, let
\beq
\widetilde{\cF}(t)=\sum_{u \in \cU_{p}} f_{B,u}(t^{u^{-1}})
\eeq
and
\beq
\cF^{b}(t)=\sum_{a \in b\cD_{1}^{'}} \chi (ab^{-1})\widetilde{\cF}|[a](t^{\langle ab^{-1}\rangle}).
\eeq
Here $|[a]$ is the Hecke action induced by $a$.\\
\\
Let $L_{p}^{b}(f,\chi)$ be the restriction of the measure $L_{p}(f,\chi)$ to $\pi_{-}(b)\Gamma^{'}$. By definition, we have
\beq
\mu(L_{p}(f,\chi))=\min_{b \in \cD_{1}^{''}} \mu(L_{p}^{b}(f,\chi)). 
\eeq
In view of the Lemma 5.5, we have
\beq
d^{q}\cF^{b}\big{|}_{t=1}=\int_{\Gamma'} \nu_{q}dL_{p}^{b}(f,\chi).
\eeq
Here $q$ is a non-negative integer and $\nu_{q}$ the $p$-adic character of $\Gamma'$ such that 
$\nu_{q}(\rec_{K}(y))=y^{q}$ for $y \in 1+p\Z_{p}$.\\
\\
It follows that the formal $t$-expansion $\cF^{b}(t)$ equals the power series expansion of the measure 
$L_{p}^{b}(f,\chi)$ regarded as a $p$-adic measure on $\Z_p$ with support in $1+p\Z_p$ (\cf (1.3)).\\
\\
Note that 
\beq 
\widetilde{\cF}(t)=|\cU^{alg}| \cdot \sum_{u \in \cU_{p}/\cU^{alg}} f_{B,u}(t^{u^{-1}}). 
\eeq
Thus, we have
\beq
\cF^{b}(t)=|\cU^{alg}| \cdot\sum_{(u,a)\in \cD_{0} \times b\cD_{1}^{'}}\chi (ab^{-1})f_{B,u}|[a](t^{\langle ab^{-1}\rangle u^{-1}}).
\eeq
Note that $p \ndivide |\cU^{alg}|$. 
From \cite[Lemma 5.3]{Hs1}, the linear independence of mod $p$ modular forms (cf. Theorem 4.1) 
and the $t$-expansion principle of \padic modular forms, it follows that
\beq
\mu(L_{p}^{b}(f,\chi))=\min_{u \in \cD_{0}}\mu(f_{B, u}).
\eeq
In view of Proposition 5.3 and Lemma 5.4, this finishes the proof.\\
\end{proof}
\noindent
\begin{remark} 
(1). The hypothesis that the Hecke character $\chi$ is unramified is present in \cite{Br}. 
It mainly arises as Prasanna's explicit version of the Waldspurger formula in \cite{Pr} is conditional on the hypothesis. 
It seems likely that the hypothesis can be removed from the above theorem once we have an explicit version of the 
Waldspurger formula in the ramified case. 
Under mild hypotheses, such a Waldspurger formula is perhaps available (\cf \cite{YZZ}).\\
\\
(2). The theorem has a similar flavour as the results on the vanishing of the $\mu$-invariant in \cite{V1} and \cite{Hs3}. 
In these articles, the hypothesis (irr) is essential for the vanishing. 
For a discussion of the necessity, we refer to the introduction of these articles.\\
\\
(3). A closely related $p$-adic L-function is constructed in \cite{LZZ}. Our strategy 
applies to this $p$-adic L-function as well and 
we can deduce the vanishing of its $\mu$-invariant under the above hypotheses.\\
\\
(4). The theorem can be used as an input in the proof of Perrin-Riou's conjecture on Heegner points under mild hypotheses in \cite{W}. 
The result on the $\mu$-invariant in \cite{Hs3} is originally used in \cite{W}.\\
\end{remark}
\noindent\\
\\
\section{Non-triviality of $p$-adic Abel-Jacobi image modulo $p$} 
\noindent In this section, we consider the non-triviality of the $p$-adic Abel-Jacobi image of generalised Heegner cycles modulo $p$. 
In \S6.1, we describe the $p$-adic Waldspurger formula due to Brooks. In \S6.2, we prove the non-triviality.\\
\\
\subsection{$p$-adic Waldspurger formula} 
In this subsection, we describe the $p$-adic Waldspurger formula due to Brooks relating certain values of the anticyclotomic Rankin-Selberg $p$-adic L-function outside the range of interpolation 
to the $p$-adic Abel-Jacobi image of generalised Heegner cycles.\\
\\
Unless otherwise stated, let the notation and hypotheses be as in the introduction. 
Let $f$ be a normalised elliptic newform of even weight $k\geq 2$, level $\Gamma_{0}(N)$ and neben-character $\epsilon$. 
We also denote the Jacquet-Langlands transfer as in \S5.1 by the same notation. 
Let $\omega_{f}$ be the corresponding differential.\\
\\ 
We first recall that the construction of generalised Heegner cycles in \cite{Br} requires the weight being even (\cf \cite[\S6]{Br}).\\
%Let $\Sg_{cc}(c,\mathfrak{N})$ be the set of Hecke characters $\lam$ satisfying (C1)', (C2) and (C3).\\
\\
Let 
\beq
r=\frac{k-2}{2}.
\eeq
\noindent
Let $A$ be the CM abelian surface corresponding to the trivial ideal class in $\Pic(\cO)$ defined over the Hilbert class field $H$ of $K$ (\cf \S2.4). 
%Let $Isog^{\mathfrak{N}}(A)$ be a subset of isogenies of $A$ as in [loc. cit.,?]. 
%For an ideal $\mathfrak{a} \subset \cO_{c}$, let $\varphi_{\mathfrak{a}} \in Isog^{\mathfrak{N}}(A)$ be as in [loc. cit., ?]. 
Let $\cA_r$ be the Kuga-Sato variety given by $r$-fold fiber product of the universal abelian surface.
For an extension 
$F/K$ containing the real quadratic field $M_{0}$ (\cf \S2.1), let $W_r$ be the variety over $F$ given by $W_{r}=\cA_{r} \times A^{r}$. 
By the abuse of notation, let $\epsilon$ also denote the idempotent in the ring of correspondences on $W_r$ defined in \cite[\S6.1]{Br}.\\
\\
%Let $g$ be a cusp form of weight $k$ defined over $F$ and $\omega_{g}$ the corresponding differential. 
For an integer $j$ such that $0 \leq j \leq 2r$, let $\omega_{f} \wedge \omega_{A}^{j}\eta_{A}^{2r-j} \in Fil^{2r+1}\epsilon H^{4r+1}_{dR}(W_r/F)$ be as in \cite[\S6.4]{Br}. 
%For $\varphi \in Isog^{\mathfrak{N}}(A)$, 
Let $n$ be a positive integer. 
For an ideal $\mathfrak{a} \subset \cO_{N^{+}p^{n}}$, 
let $\Delta_{\mathfrak{a}} \in \epsilon CH^{2r+1}(W_{r} \otimes L)_{0,\Q}$ be the codimension-$(2r+1)$ homologous to zero generalised Heegner cycle defined in \cite[\S6.2]{Br}. 
Here $L$ is the field of definition of the cycle and $CH^{2r+1}(W_{r} \otimes L)_{0,\Q}$ is the Chow group of codimension-$(2r+1)$ homologous to zero cycles over $L$ 
with rational coefficients.\\
\\
Let
\beq
AJ_{p}: \epsilon CH^{2r+1}(W_{r})_{0,\Q} \rightarrow (Fil^{2r+1}\epsilon H^{4r+1}_{dR}(W_r/F)(r))^{\vee}
\eeq
be the $p$-adic Abel-Jacobi map in \cite[\S6.3]{Br}.\\ 
%Here $(Fil^{r+1}\epsilon_{X}H^{2r+1}_{dR}(X_r/F)(r))^{\vee}$ is the dual of 
%$Fil^{r+1}\epsilon_{X}H^{2r+1}_{dR}(X_r/F)(r)$.\\
\\
We have the following $p$-adic Waldspurger formula.\\
\begin{thm}(Brooks) 
Let the notation be as above. 
Let $f \in S_{k}(\Gamma_{0}(N), \epsilon)$ be an elliptic newform 
and $\chi \in \Sg_{cc}^{(2)}$ an unramified Hecke character over $K$ with infinity type $(k,0)$. 
Let $\eta$ be a Hecke character such that $\chi\eta \in \Sg_{cc}^{(1)}$ 
is an unramified Hecke character over $K$ with infinity type $(k-1-j,1+j)$ for $0 \leq j \leq 2r$.
Let $\nu \in \mathfrak{X}_{0}$ be a primitive Hecke character of conductor $p^{n}$, where $n\geq 1$. 
Suppose that the hypotheses (ord), (H1), (H2) and (H3) hold. 
Then, we have\\
$$\frac{\widehat{\eta\nu}(L_{p}(f,\chi))}{\Omega_{p}^{2(2r-2j)}} =
\bigg{(} \frac{G(\nu^{-1})}{j!} \cdot\sum_{[\mathfrak{a}]\in G_{p^{n}}} (\chi\eta\nu)^{-1}(\mathfrak{a}){\bf{N_{K}}}(\mathfrak{a})\cdot AJ_{p}(\Delta_{\mathfrak{a}})(\omega_{f} \wedge \omega_{A}^{j}\eta_{A}^{r-j})\bigg{)}^{2}.$$
\noindent\\
\end{thm}
\begin{proof}
This follows from the argument in the proof of Theorem 8.11 in \cite{Br}. 
As $\eta$ is of $p$-power conductor, we obtain the extra factor of the Gauss sum. 
For a related argument, we refer to the proof of \cite[Thm. 4.9]{Ca}.\\
\end{proof}
\noindent\\
\\ 
\subsection{Non-triviality} 
In this subsection, we consider the non-triviality of the $p$-adic Abel-Jacobi image of generalised Heegner cycles modulo $p$.\\
\\
Let the notation and hypotheses be as in \S6.1. 
We have the following result regarding the non-triviality.\\
\begin{thm}
Let $f \in S_{k}(\Gamma_{0}(N),\epsilon)$ be an elliptic newform of even weight
and $\chi \in \Sg_{cc}^{(2)}$ an unramified Hecke character over $K$ with infinity type $(k,0)$. 
Let $\eta$ be a Hecke character such that $\chi\eta \in \Sg_{cc}^{(1)}$ 
is an unramified Hecke character over $K$ with infinity type $(k-1-j,1+j)$ for $0 \leq j \leq 2r$.
Suppose that the hypotheses (ord), (H1), (H2), (H3) and (irr) hold.
Then, we have\\
$$\liminf_{\nu \in \mathfrak{X}_{0}} v_{p}   \bigg{(}\frac{G(\nu^{-1})}{j!} \cdot
\sum_{[\mathfrak{a}]\in G_{p^{n}}} (\chi\eta\nu)^{-1}(\mathfrak{a}){\bf{N_{K}}}(\mathfrak{a})\cdot AJ_{p}(\Delta_{\mathfrak{a}})(\omega_{f} \wedge \omega_{A}^{j}\eta_{A}^{r-j}) \bigg{)}=0, $$
where $p^{n}$ is the conductor of $\nu$. 
Moreover, the same conclusion holds when $\mathfrak{X}_{0}$ is replaced by any of its infinite subset.\\
\end{thm}
\begin{proof}
This follows readily from Theorem 5.7 and Theorem 6.1.\\
\end{proof}
\noindent
\begin{remark} 
(1). It follows that the generalised Heegner cycles are non-trivial in the top graded piece of the coniveau filtration on the Chow group over the $\Z_p$-anticyclotomic extension. 
%Moreover, they are $p$-indivisible. 
The non-triviality can be seen as an evidence for the refined Bloch-Beilinson conjecture as follows. 
Recall that the Rankin-Selberg convolution in consideration is self-dual with root number $-1$. 
In view of \cite{Sc} and the Jacquet-Langlands correspondence, the corresponding Galois representation contributes to an étale cohomology of $W_r$. 
The conjecture thus predicts the existence of a non-trivial cycle in the top graded piece of the coniveau filtration on the Chow group. 
Generalised Heegner cycles are a natural source of cycles in the setup and in the case of weight two, they coincide with classical Heegner points. 
We can thus expect a generic non-triviality of generalised Heegner cycles. 
For the details and the role of coniveau filtration, we refer to \cite[\S1 and \S2]{BDP3} and \cite{Bu4}.\\
\\
(2). In view of the theorem and the construction of generalised Heegner cycles, it follows that 
the Griffiths group $Gr^{r+1}(W_{r/\overline{\Q}}) \otimes \Q$ has infinite rank. 
An analogous result for the Griffiths group of the Kuga-Sato variety $\cA_r$ is due to Besser (\cf \cite{Be}). 
The approach in \cite{Be} is via consideration of generic non-triviality of classical Heegner cycles over a class of varying imaginary quadratic extensions.\\
\end{remark}
\noindent\\
We have the following immediate corollary.\\
\begin{cor}
Let $f \in S_{2}(\Gamma_{0}(N), \epsilon)$ be an elliptic newform and 
$\chi$ be an unramified finite order Hecke character over $K$ such that $\chi {\bf{N_{K}}} \in \Sg_{cc}^{(1)}$. 
Suppose that the hypotheses (ord), (H1), (H2), (H3) and (irr) hold. 
Then, we have\\
$$\liminf_{\nu \in \mathfrak{X}_{0}}v_{p}\bigg{(}G(\nu^{-1})\log_{\omega_{B_{f}}}(P_{f}(\chi \nu))\bigg{)}=0.$$
\\
In particular,  
for $\nu \in \mathfrak{X}_{0}$ with sufficiently large $p$-power order the Heegner points $P_{f}(\chi\nu)$ are non-zero in $B_{f}(H_{\chi\nu}) \otimes_{T_f} E_{f,\eta\nu}$. 
Moreover, the same conclusions hold when $\mathfrak{X}_{0}$ is replaced by any of its infinite subset.\\
\end{cor}
\begin{proof} 
This follows from the weight 2 case of the theorem.\\
\end{proof}
\noindent\\

\end{document}